\documentclass[preprint]{elsarticle}

\usepackage{graphicx}
\usepackage{amssymb}
\usepackage{lineno}
\usepackage{hyperref}
\usepackage{amsmath}
\usepackage{amsthm}
\usepackage{stmaryrd}
\usepackage{multirow}
\usepackage{enumitem}
\usepackage{dsfont}
\usepackage{makecell}
\usepackage{soul}
\usepackage{appendix}
\usepackage{multirow}
\usepackage[dvipsnames]{xcolor}
\usepackage{arydshln}
\setuldepth{Paris}

\newtheoremstyle{theorems}{1em}{1em}{}{}{\bfseries}{}{ }{}
\theoremstyle{theorems}
\newtheorem{theorem}{Theorem}[section]
\newtheorem{proposition}{Proposition}[section]
\newtheorem{definition}{Definition}[section]
\newtheorem{lemma}{Lemma}[section]

\definecolor{lightblue}{rgb}{0.8,0.8,1}
\sethlcolor{lightblue}

\newcommand{\intg}[2]{\{#1,...,#2\}}
\newcommand{\N}{\mathbb{N}}
\newcommand{\R}{\mathbb{R}}

\newcommand{\M}{\mathcal{M}}
\newcommand{\mc}{\mathcal}
\newcommand{\mf}{\mathfrak}
\newcommand{\lto}{\longrightarrow}
\newcommand{\lmto}{\longmapsto}

\newcommand{\cns}{\Longleftrightarrow}
\newcommand{\ls}{\leqslant}
\newcommand{\gs}{\geqslant}

\newcommand{\tr}{\mathrm{tr}}

\newcommand{\inv}{\mathrm{inv}}
\newcommand{\eig}{\mathrm{eig}}
\newcommand{\pow}{\mathrm{pow}}
\newcommand{\diag}{\mathrm{diag}}
\newcommand{\Diag}{\mathrm{Diag}}

\newcommand{\Id}{\mathrm{Id}}
\newcommand{\id}{\mathrm{id}}

\newcommand{\Mat}{\mathrm{Mat}}
\newcommand{\Sym}{\mathrm{Sym}}

\newcommand{\Skew}{\mathrm{Skew}}

\newcommand{\GL}{\mathrm{GL}}
\newcommand{\Orth}{\mathrm{O}}
\newcommand{\SO}{\mathrm{SO}}

\newcommand{\SPD}{\mathrm{SPD}}

\newcommand{\Exp}{\mathrm{Exp}}
\newcommand{\Log}{\mathrm{Log}}

\newcommand{\dotprod}[2]{\langle #1|#2\rangle}
\newcommand{\mini}{\mathrm{min}}
\newcommand{\maxi}{\mathrm{max}}

\newcommand{\Frob}{\mathrm{Frob}}
\newcommand{\E}{\mathrm{E}}

\newcommand{\LE}{\mathrm{LE}}
\newcommand{\A}{\mathrm{A}}
\newcommand{\BW}{\mathrm{BW}}

\newcommand{\BKM}{\mathrm{BKM}}

\newcommand{\ver}{\mathrm{ver}}
\newcommand{\hor}{\mathrm{hor}}

\newcommand{\twoeq}[2]{
\left\{
\begin{array}{ccc}
#1 \\ \relax
#2 \\
\end{array}
\right.
}

\newcommand{\fun}[4]{
\left\{
\begin{array}{ccc}
#1 & \lto & #2\\ \relax
#3 & \lmto & #4\\
\end{array}
\right.
}

\newcommand{\sys}[4]{
\left\{
\begin{array}{ccc}
#1 & \mathrm{if} & #2\\ \relax
#3 & \mathrm{if} & #4\\
\end{array}
\right\}
}

\newcommand{\accol}[3]{
\left\{
\begin{array}{l}
#1\\ \relax
#2\\ \relax
#3\\
\end{array}
\right.}

\journal{Linear Algebra and its Applications}

\begin{document}

\begin{frontmatter}

\title{$\Orth(n)$-invariant Riemannian metrics on SPD matrices}
\author{Yann Thanwerdas\corref{cor1}}
\ead{yann.thanwerdas@inria.fr}
\author{Xavier Pennec}
\ead{xavier.pennec@inria.fr}

\cortext[cor1]{Corresponding author}

\address{Université Côte d'Azur and Inria, Epione Project Team, Sophia Antipolis\\2004 route des Lucioles, 06902 Valbonne Cedex, France
}

\begin{abstract}
Symmetric Positive Definite (SPD) matrices are ubiquitous in data analysis under the form of covariance matrices or correlation matrices. Several $\Orth(n)$-invariant Riemannian metrics were defined on the SPD cone, in particular the kernel metrics introduced by Hiai and Petz. The class of kernel metrics interpolates between many classical $\Orth(n)$-invariant metrics and it satisfies key results of stability and completeness. However, it does not contain all the classical $\Orth(n)$-invariant metrics. Therefore in this work, we investigate super-classes of kernel metrics and we study which key results remain true. We also introduce an additional key result called cometric-stability, a crucial property to implement geodesics with a Hamiltonian formulation. Our method to build intermediate embedded classes between $\Orth(n)$-invariant metrics and kernel metrics is to give a characterization of the whole class of $\Orth(n)$-invariant metrics on SPD matrices and to specify requirements on metrics one by one until we reach kernel metrics. As a secondary contribution, we synthesize the literature on the main $\Orth(n)$-invariant metrics, we provide the complete formula of the sectional curvature of the affine-invariant metric and the formula of the geodesic parallel transport between commuting matrices for the Bures-Wasserstein metric.
\end{abstract}


\end{frontmatter}


\section{Introduction}

Symmetric Positive Definite (SPD) matrices are ubiquitous in data analysis because in many situations, the data (signals, images, diffusion coefficients...) can be represented by their covariance matrices. This is the case in the domains of Brain-Computer Interfaces, diffusion and functional MRI, Computer Vision, Diffusion Tensor Imaging (DTI)... SPD matrices form a cone in the vector space of symmetric matrices so a first idea to compute with SPD matrices could be to perform Euclidean computations on symmetric matrices. However, this method has several drawbacks. As geodesics are straight lines, they leave the SPD cone at finite time so extrapolation methods could lead to non admissible matrices, namely with negative eigenvalues. Moreover, the trace is linearly interpolated but other invariants such as the determinant are not monotonically interpolated along geodesics. For example in DTI, where SPD matrices are represented by 3D ellipsoids, the ellipsoids along the geodesic can have a larger volume than the two ellipsoids at extremities, which leads to non realistic predictions in fiber tracking (swelling effect).

Hence, other Riemannian metrics were used in applications to solve these problems. The affine-invariant/Fisher-Rao metric \cite{Skovgaard84,Pennec06,Lenglet06-JMIV,Fletcher07,Moakher05,Batchelor05,Varoquaux10,Barachant13} provides a Riemannian symmetric structure to the SPD manifold: it is negatively curved, geodesically complete (matrices with null eigenvalues are rejected to infinity), it is invariant under the congruence action (which, in the context of covariance matrices, corresponds to the invariance of the feature vector under affine transformations) and it is inverse-consistent. The log-Euclidean metric \cite{Arsigny06} is diffeomorphic to a Euclidean inner product: it also provides a Riemannian symmetric space, it is geodesically complete and inverse-consistent. It is not curved and it is not affine-invariant although it is still invariant under orthogonal and dilation transformations. The Bures-Wasserstein/Procrustes metric \cite{Bhatia19,Dryden09,Takatsu11,Malago18} is a positively curved quotient metric which is also invariant under orthogonal transformations. It is not geodesically complete but geodesics remain in the cone with boundaries: this means that this metric is suited for computing with Positive Semi-Definite (PSD) matrices. Many other interesting metrics exist with different properties: Bogoliubov-Kubo-Mori \cite{Petz93,Michor00}, polar-affine \cite{Su12}, Euclidean-Cholesky \cite{Wang04}, log-Euclidean-Cholesky \cite{Li17}, log-Cholesky \cite{Lin19}, power-Euclidean \cite{Dryden10}, and more recently power-affine \cite{Thanwerdas19-b}, alpha-Procrustes \cite{HaQuang19}, mixed-power-Euclidean \cite{Thanwerdas19-a}.

Except those named after Cholesky, all the other Riemannian metrics cited above are invariant under orthogonal transformations. If we consider SPD matrices as covariance matrices, this transformation corresponds to a rigid-body transformation of the feature vector $X\in\R^n\lmto RX+X_0$ where $R$ is an orthogonal matrix. In 2009, Hiai and Petz introduced the subclass of \textit{kernel metrics} \cite{Hiai09}, which are $\Orth(n)$-invariant metrics indexed by smooth symmetric maps $\phi:(0,\infty)^2\lto(0,\infty)$. This class satisfies key results: it contains most of the cited $\Orth(n)$-invariant metrics, it is stable under a certain class of diffeomorphisms and it provides a sufficient condition for geodesic completeness. This sufficient condition becomes necessary if we restrict the class to the subclass of \textit{mean kernel metrics} which is indexed by kernel maps of the form $\phi=m^\theta$ where $m:(0,\infty)^2\lto(0,\infty)$ is a symmetric homogeneous mean and $\theta\in\R$ is a power. However, the class of kernel metrics does not contain \textit{all} the aforementioned $\Orth(n)$-invariant metrics. The main goal of this paper is to study the super-classes of kernel metrics, especially the whole class of $\Orth(n)$-invariant metrics for which we give a characterization. More precisely, our objective is to determine which key results on kernel metrics can be generalized and thus to understand better the specificity of kernel metrics within these super-classes.

\subsection{Results and organization of the paper}

In the remainder of the Introduction, we give the notations and conventions used in the paper. In Section \ref{sec:preliminary_concepts}, we introduce two preliminary concepts and one result. The first concept is the notion of \textit{$\Orth(n)$-equivariant map} on symmetric matrices. We especially explain how to build them from a map defined on diagonal matrices via the spectral theorem because this is a procedure we need several times in the paper. Then the second concept is a particular case of the previous one, called \textit{univariate map}. These are maps characterized by a map on positive real numbers. They are particularly interesting because their differential is known in closed form modulo eigenvalue decomposition and because the class of kernel metrics is stable under univariate diffeomorphisms. Finally the result is the characterization of $\Orth(n)$-invariant inner products on symmetric matrices. These inner products are composed of two terms, the Frobenius term and the trace term, which have different weights so they form a two-parameter family. In the proof, we give elementary tools which can be reused when we look for the characterization of $\Orth(n)$-invariant metrics on SPD matrices.

To explain why kernel metrics do not encompass all the $\Orth(n)$-invariant metrics cited above, we need to present them or at least the most important ones. One can notice that many metrics and families of metrics are actually based on five of them, namely the Euclidean, the log-Euclidean, the affine-invariant, the Bures-Wasserstein and the Bogoliubov-Kubo-Mori metrics. That is why in Section \ref{sec:inventory}, we synthesize the literature on these five main metrics. For each of them, we give the fundamental Riemannian operations (squared distance, Levi-Civita connection, curvature, geodesics, logarithm map, parallel transport map) when they are known. As a secondary contribution of the paper, we give the complete formula of the sectional curvature of the affine-invariant metric and we also give, for the Bures-Wasserstein metric, the new formula of the parallel transport between commuting matrices and simpler formulae of the Levi-Civita connection, the curvature and the parallel transport equation.

In Section \ref{sec:kernel}, after reviewing kernel metrics and their key properties, we give two new main observations on them. Firstly, the cometric of a metric on SPD matrices can be considered itself as a metric on SPD matrices by identifying the vector space of symmetric matrices and its dual via the Frobenius inner product. Therefore we observe that the cometric of a kernel metric defined by the kernel map $\phi$ is a kernel metric characterized by $1/\phi$. This remarkable result has an important consequence for the numerical computation of geodesics. Indeed, the geodesic equation $\nabla_{\dot\gamma}\dot\gamma=0$, which is a second order equation, has a Hamiltonian version which is a first order equation that only involves the cometric, not the Christoffel symbols. The Hamiltonian equation is much simpler to integrate and numerically more stable, that is why it is often preferred in numerical implementations, for instance in the \href{https://geomstats.github.io/}{Python package geomstats} \cite{Miolane20-JMLR}. Hence knowing a simple explicit formula for the cometric helps to compute numerically the geodesics. Secondly, there is a natural extension of kernel metrics that encompasses all the aforementioned $\Orth(n)$-invariant metrics, which still satisfies the key properties of kernel metrics including the cometric stability. Roughly speaking, kernel metrics look like the Frobenius inner product on symmetric matrices where the elementary quadratic forms (the $X_{ij}^2$) are weighted by a coefficient involving the kernel map $\phi$ and depending on the point. Since the Frobenius inner product is not the only $\Orth(n)$-invariant inner product on symmetric matrices as explained above, the trace term can be added to the framework of kernel metrics to form extended kernel metrics.

In Section \ref{sec:characterization}, we characterize the class of $\Orth(n)$-invariant metrics on SPD matrices by means of three multivariate maps $\alpha,\beta,\gamma:(0,\infty)^n\lto\R$ operating on the eigenvalues $(d_1,...,d_n)$ of the SPD matrix and which satisfy three conditions of compatibility, positivity and symmetry (Theorem \ref{thm:characterization_o(n)-invariant_inner_products}). Then, we observe that kernel metrics are characterized by two properties within this family. They are \textit{ortho-diagonal}: it means that the metric matrix is diagonal, i.e. $\beta=0$. They are \textit{bivariate}: it means that the remaining functions $\alpha$ and $\gamma$ do not depend on their $n-2$ last terms, and the compatibility condition imposes that they are equal so we can write $\gamma=\alpha=1/\phi:(0,\infty)^2\lto(0,\infty)$. Since the term ``kernel" is quite overloaded in many different contexts (such as in Reproducing Kernel Hilbert Spaces in machine learning or in kernel density estimation/regression in statistics), we propose to designate them as Bivariate Ortho-Diagonal (BOD) metrics. Afterwards, we give key properties of $\Orth(n)$-invariant metrics in analogy with the key properties of BOD (kernel) metrics. In particular, we do not have a closed-form expression for the cometric anymore. To solve this problem, we introduce the intermediate class of bivariate separable metrics which is cometric-stable and we give the expression of the cometric. A summary of the classes of metrics defined in the paper is shown on Figure \ref{fig:super-classes_of_kernel_metrics}.

Section \ref{sec:conclusion} is dedicated to the conclusion.

\begin{figure}[htbp]
    \centering
    \includegraphics[scale=0.4]{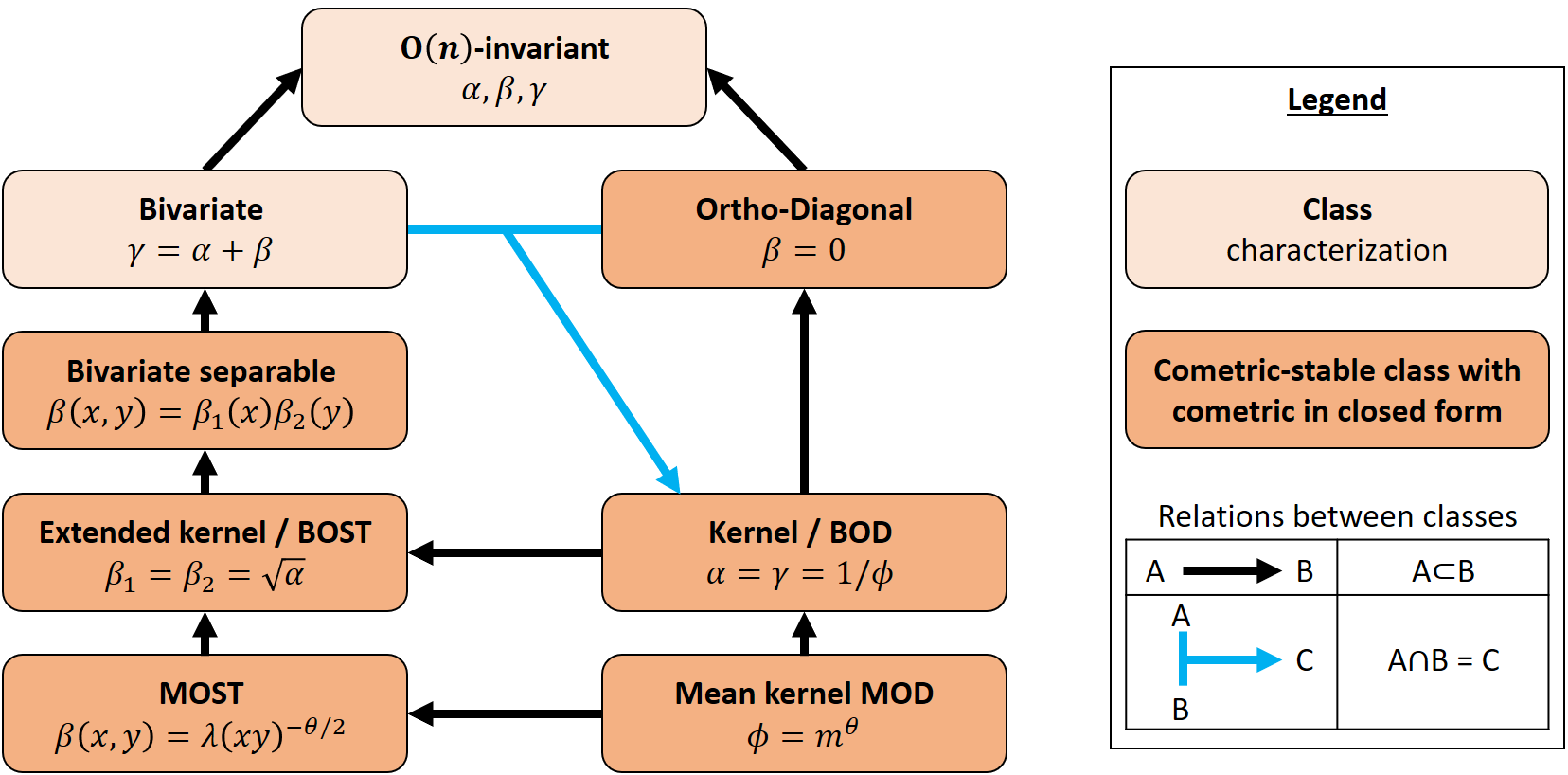}
    \caption{Super-classes of kernel metrics}
    \label{fig:super-classes_of_kernel_metrics}
\end{figure}

\subsection{Notations and conventions}

\paragraph{Manifolds}
Our manifold-related notations are summarized in Table \ref{table_notations_manifold}. A chart $\varphi:\mc{U}\subset\mc{M}\lto\R^N$ provides a local basis of vectors $(\partial_1,...,\partial_N)$ where $\partial_k=\frac{\partial}{\partial\varphi^k}$ is a short notation defined for all differentiable maps $f:\mc{M}\lto\R$ and at each point $x\in\mc{U}$ by $(\partial_kf)_{|x}=\left.\frac{\partial(f\circ\varphi^{-1})}{\partial x^k}\right|_{\varphi(x)}$.
A vector field $X$ can be locally decomposed on this basis, $X=X^k\partial_k$, where $X^k:\mc{U}\lto\R$ are the coordinate functions of $X$ and where we used Einstein's summation convention. As we deal with matrices in this paper, the coordinates often have two indices: $X=X^{ij}\partial_{ij}$.

\begin{table}[htbp]
    \centering
    \begin{tabular}{|c|c|}
    \hline
    $T_x\M,T\M$ & Tangent space at $x$, tangent bundle \\
    \hline
    $d_xf,df$ & Differential of map $f$ at $x$, differential of map $f$ \\
    \hline
    $f^*,f_*$ & Pullback via $f$, pushforward via $f$\\
    \hline
    $\dot\gamma$ & Derivative of curve $\gamma$\\
    \hline
    $g,G$ & Metric on $\SPD(n)$, metric on another space \\
    \hline
    $d$ & Riemannian distance on $\SPD(n)$\\
    \hline
    $\nabla$ & Levi-Civita connection\\
    \hline
    $R$ & Curvature $R(X,Y)Z=\nabla_X\nabla_YZ-\nabla_Y\nabla_XZ-\nabla_{[X,Y]}Z$ \\
    \hline
    $\gamma_{(\Sigma,X)}(t)$ & Geodesic at time $t$ with $\gamma(0)=\Sigma$ and $\dot\gamma(0)=X$ \\
    \hline
    $\Exp,\Log$ & Riemannian exponential and logarithm maps\\
    \hline
    $\Pi_{\gamma;\Sigma\to\Lambda}X$ & Parallel transport of $X$ along curve $\gamma$ from $\Sigma$ to $\Lambda$ \\
    \hline
\end{tabular}
    \caption{Notations in a manifold}
    \label{table_notations_manifold}
\end{table}

\paragraph{Manifolds of matrices}
We denote the vector spaces, Lie groups and manifolds of matrices as shown in Table \ref{table_notations_matrices}. The $(i,j)$-coefficient of a matrix $M$ is denoted $M_{ij}$ or $[M]_{ij}$ or $M(i,j)$ depending on the context, for readability. To build a matrix from its coefficients, we denote $M=[M_{ij}]_{1\ls i,j\ls n}$ or simply $M=[M_{ij}]_{i,j}$. We denote $(C_{ij})$ the canonical basis of matrices, $E_{ii}=C_{ii}$, $E_{ij}=\frac{1}{\sqrt{2}}(C_{ij}+C_{ji})$ and $F_{kl}=\frac{1}{2}(C_{kl}+C_{lk})$ for $i\ne j$ and $k,l\in\intg{1}{n}$.

\begin{table}[htbp]
\centering
\begin{tabular}{|c|c||c|c|}
    \hline
    \multicolumn{2}{|c||}{Vector space of matrices} & \multicolumn{2}{c|}{Manifold of matrices}\\
    \hline
    $\Mat(n)$ & $n\times n$ real matrices & $\GL(n)$ & General Linear group\\
    & & $\GL^+(n)$ & Positive determinant\\
    \hline
    $\Sym(n)$ & Real symmetric & $\SPD(n)$ & Symmetric positive definite\\
    \hline
    $\Skew(n)$ & Real skew-symmetric & $\Orth(n)$ & Orthogonal group\\
    & & $\SO(n)$ & Rotation group\\
    \hline
    $\Diag(n)$ & Diagonal & $\Diag^+(n)$ & Positive diagonal\\
    \hline
\end{tabular}
\caption{Notations for matrix spaces}
\label{table_notations_matrices}
\end{table}

The congruence action is the following action of the general linear group on matrices $(A,M)\in\GL(n)\times\Mat(n)\lmto AMA^\top\in\Mat(n)$ which leaves stable the spaces of symmetric matrices and SPD matrices.

The symmetric group of order $n$ is denoted by $\mf{S}_n$ and the permutations by small greek letters $\sigma,\tau...$. The permutation matrix associated to the permutation $\sigma$, which sends any basis $(e_1,...,e_n)$ of $\R^n$ to the permuted basis $(e_{\sigma(1)},...,e_{\sigma(n)})$, is denoted $P_\sigma$. We have $P_\sigma(i,j)=\delta_{\sigma(i),j}$ where $\delta$ is the Kronecker symbol. Given a matrix $M\in\Mat(n)$, we have $(P_\sigma^\top MP_\sigma)(i,j)=M(\sigma(i),\sigma(j))$.

\paragraph{The manifold of SPD matrices}
The manifold $\SPD(n)$ is an open set of the vector space of symmetric matrices $\Sym(n)$. Hence, the canonical immersion $\id:\SPD(n)\hookrightarrow\Sym(n)$ provides:
\begin{enumerate}[label=$\bullet$]
    \itemsep0em
    \item An identification between the tangent space $T_\Sigma\SPD(n)$ and the vector space $\Sym(n)$ at any point $\Sigma\in\SPD(n)$ by $d_\Sigma\id:T_\Sigma\SPD(n)\overset{\sim}{\lto}\Sym(n)$. Thus, any tangent vector $X\in T_\Sigma\SPD(n)$ is considered as a symmetric matrix: $X\equiv d_\Sigma\id(X)\in\Sym(n)$.
    \item A global chart $(\id,\SPD(n))$ of the manifold $\SPD(n)$, thus a global derivation $\partial_XY=X^{ij}(\partial_{ij}Y^{kl})\partial_{kl}$ defined by derivation of coordinates in this global chart. More generally, if $f:\SPD(n)\lto\Sym(n)$ is a diffeomorphism on its image, it provides a global derivation that we denote $\partial^f$.
\end{enumerate}

Another important tool is the matrix exponential $\exp(X)=\sum_{k=0}^{+\infty}{\frac{X^k}{k!}}$ which is a diffeomorphism between $\Sym(n)$ and $\SPD(n)$, and therefore its inverse, the symmetric matrix logarithm $\log:\SPD(n)\lto\Sym(n)$.

The spectral theorem ensures that symmetric matrices are orthogonally congruent to a diagonal matrix. If the symmetric matrix is SPD, then the diagonal matrix has positive elements on the diagonal. Most of the time in this paper, for an SPD matrix $\Sigma\in\SPD(n)$, we denote $\Sigma=PDP^\top$ one spectral decomposition with $P\in\Orth(n)$ and $D=\diag(d_1,...,d_n)\in\Diag^+(n)$. When we consider tangent vectors $X,Y,...\in T_\Sigma\SPD(n)$, we denote $X'=P^\top XP$ so that every matrix expressed in the orthogonal basis given by $P$ is denoted with a prime: $X=PX'P^\top$, $Y=PY'P^\top$, ...

Products of symmetric matrices share two nice properties with symmetric matrices. First, if $X,Y\in\Sym(n)$, then $\eig(XY)\subset\R$ where $\eig$ denotes the set of complex eigenvalues. Second, if $\Sigma,\Lambda\in\SPD(n)$, then $\Sigma\Lambda$ has a unique square-root matrix that represents a positive definite self-adjoint endomorphism, it is denoted $(\Sigma\Lambda)^{1/2}=\sqrt{\Sigma\Lambda}=\Sigma^{1/2}(\Sigma^{1/2}\Lambda\Sigma^{1/2})^{1/2}\Sigma^{-1/2}=\Lambda^{-1/2}(\Lambda^{1/2}\Sigma\Lambda^{1/2})^{1/2}\Lambda^{1/2}$.

\section{Preliminary concepts and results}\label{sec:preliminary_concepts}

\subsection{$\Orth(n)$-equivariant maps}

In our context of SPD matrices, we call $\Orth(n)$-equivariant map a map $f:\SPD(n)\lto\Sym(n)$ such that $f(R\Sigma R^\top)=R\,f(\Sigma)\,R^\top$ for all $\Sigma\in\SPD(n)$ and $R\in\Orth(n)$. Thanks to the spectral theorem, they are characterized by their values on positive diagonal matrices. A question that arises several times in this paper is: are we allowed to extend a map $f:\Diag^+(n)\lto\Sym(n)$ into an $\Orth(n)$-equivariant map $f:\SPD(n)\lto\Sym(n)$ by the formula $f(PDP^\top)=P\,f(D)\,P^\top$? To do so, we have to show that given two eigenvalue decompositions $\Sigma=PDP^\top=Q\Delta Q^\top$, then $P\,f(D)\,P^\top=Q\,f(\Delta)\,Q^\top$. Note that $(Q,\Delta)$ is highly constrained by $(P,D)$. The following lemma gives explicitly the possible cases, hence it tells exactly what is to be checked in such an extension process. We omit the proof.

\begin{lemma}
[Relation between two eigenvalue decompositions of an SPD matrix]\label{lemma:spectral_dec}
Let $D,\Delta$ be positive diagonal matrices and $R\in\Orth(n)$ be an orthogonal matrix. Without loss of generality, we assume that $D=\Diag(\lambda_1I_{m_1},...,\lambda_pI_{m_p})$ with $\lambda_1>...>\lambda_p>0$ and respective multiplicities $m_1,...,m_p$.
\begin{enumerate}[label=(\alph*)]
    \itemsep0em
    \item For all $\varepsilon=\Diag(\pm1,...,\pm1)$, $D=\varepsilon D\varepsilon$. \label{enum:sign}
    \item If $D=R\Delta R^\top$, then there exists a permutation $\sigma\in\mf{S}_n$ s.t. $D=P_\sigma\Delta P_\sigma^\top$.
    \item If $D=RDR^\top$, then $R=\Diag(R_1,...,R_p)\in\Orth(n)$ is a block-diagonal orthogonal matrix with $j$-th block $R_j\in\Orth(m_j)$. (It contains case \ref{enum:sign}.)
\end{enumerate}
\end{lemma}

Hence, to extend $f$ from $\Diag^+(n)$ to $\SPD(n)$, it suffices to show that for all diagonal matrices $D=\Diag(\lambda_1I_{m_1},...,\lambda_pI_{m_p})$ with $\lambda_1>...>\lambda_p>0$:
\begin{enumerate}[label=(\alph*)]
    \itemsep0em
    \item $f(D)=\varepsilon\,f(D)\,\varepsilon$ for all $\varepsilon=\Diag(\pm 1,...,\pm 1)$,
    \item $f(D)=P_\sigma\,f(P_\sigma^\top DP_\sigma)P_\sigma^\top$ for all permutations $\sigma\in\mf{S}(n)$,
    \item $f(D)=R\,f(D)\,R^\top$ for all block-diagonal orthogonal matrices $R\in\Orth(n)$, $R=\Diag(R_1,...,R_p)$ with $R_j\in\Orth(m_j)$.
\end{enumerate}

\subsection{Univariate maps}\label{subsec:univariate_maps}

We apply Lemma \ref{lemma:spectral_dec} to a map defined on positive real numbers $f:(0,\infty)\lto\R$ and extended to positive diagonal matrices $f:\Diag^+(n)\lto\Diag(n)$ by $f(\Diag(d_1,...,d_n)):=\Diag(f(d_1),...,f(d_n))$.
\begin{enumerate}[label=(\alph*)]
    \itemsep0em
    \item Since $f(D)$ is diagonal, we have $f(D)=\varepsilon\,f(D)\,\varepsilon$.
    \item Since $f$ is defined component-wise, we have $f(D)=P_\sigma\,f(P_\sigma^\top DP_\sigma)P_\sigma^\top$.
    \item As $f(\lambda I_{m_j})=f(\lambda)I_{m_j}$, the matrix $Rf(D)R^\top$ is a block diagonal matrix with $j$-th block $f(\lambda_j)R_jR_j^\top=f(\lambda_j)I_{m_j}$, which corresponds to $f(D)$'s $j$-th block so $R\,f(D)\,R^\top=f(D)$.
\end{enumerate}
Therefore $f$ can be extended into an $\Orth(n)$-equivariant map $f:\SPD(n)\lto\Sym(n)$ by $f(PDP^\top)=Pf(D)P^\top$. We call these extensions univariate maps. The symmetric matrix logarithm $\log:\SPD(n)\lto\Sym(n)$, the power diffeomorphisms $\pow_p:\SPD(n)\lto\SPD(n)$ with $p\ne 0$ or the constant map $\pow_0:\Sigma\in\SPD(n)\lmto I_n\in\Sym(n)$ are examples of univariate maps.

\begin{definition}[Univariate maps]
A univariate map is the extension of a map on positive real numbers $f:(0,\infty)\lto\R$ into an $\Orth(n)$-equivariant map $f:\SPD(n)\lto\Sym(n)$ by the equality $f(PDP^\top)=P\,\Diag(f(d_1),...,f(d_n))\,P^\top$. Moreover \cite{Bhatia97}, if $f\in\mc{C}^1(0,\infty)$, then its extension $f\in\mc{C}^1(\SPD(n))$ and the differential of $f$ is $\Orth(n)$-equivariant, thus it is characterized by its values at diagonal matrices $D\in\Diag^+(n)$, given by:
\begin{equation}
    \forall X\in\Sym(n),[d_Df(X)]_{ij}=f^{[1]}(d_i,d_j)X_{ij},
\end{equation}
where $f^{[1]}$ is the first divided difference defined below.\\
The inverse function theorem ensures that a diffeomorphism $f:(0,\infty)\lto(0,\infty)$ is extended into a diffeomorphism $f:\SPD(n)\lto\SPD(n)$.
\end{definition}

\begin{definition}[First divided difference]\cite{Bhatia97}
Let $f\in\mc{C}^1((0,\infty))$. The first divided difference of $f$ is the continuous symmetric map $f^{[1]}:(0,\infty)^2\lto\R$ defined for all $x,y\in\R$ by:
    \begin{equation}
        f^{[1]}(x,y)=\sys{\frac{f(x)-f(y)}{x-y}}{x\ne y}{f'(x)}{x=y}.
    \end{equation}
\end{definition}

\subsection{$\Orth(n)$-invariant inner products on symmetric matrices}\label{subsec:o(n)-invariant_inner_products}

To characterize the $\Orth(n)$-invariant metrics on SPD matrices, an appropriate starting point is the characterization of $\Orth(n)$-invariant inner products on the tangent space, i.e. on symmetric matrices. The following theorem states that such inner products form a two-parameter family indexed by a Scaling factor $\alpha>0$ and a Trace factor $\beta>-\alpha/n$.

\begin{theorem}[Characterization of $\Orth(n)$-invariant inner products on symmetric matrices]\label{thm:characterization_o(n)-invariant_inner_products}
Let $\dotprod{\cdot}{\cdot}:\Sym(n)\times\Sym(n)\lto\R$ be an inner product on symmetric matrices. It is $\Orth(n)$-invariant if and only if there exist $(\alpha,\beta)\in\mathbf{ST}:=\{(\alpha,\beta)\in\R^2|\min(\alpha,\alpha+n\beta)>0\}$ such that:
\begin{equation}\label{eq:o(n)-invariant}
    \forall X\in\Sym(n),\dotprod{X}{X}=\alpha\,\tr(X^2)+\beta\,\tr(X)^2.
\end{equation}
Moreover, the linear isometry that pulls the Frobenius inner product back onto this one is $F_{p,q}(X)=q\,X+\frac{p-q}{n}\tr(X)I_n$ with $p=\sqrt{\alpha+n\beta}$ and $q=\sqrt\alpha$.
\end{theorem}

There are several proofs of this elementary result. We would like to give one based on the following lemma because we can reuse it to characterize $\Orth(n)$-invariant metrics on SPD matrices. This lemma gives the characterization of inner products on symmetric matrices which are respectively invariant under two subgroups of $\Orth(n)$ that we met in Lemma \ref{lemma:spectral_dec} about eigenvalue decompositions:
\begin{enumerate}[label=(\alph*)]
    \itemsep0em
    \item the group $\mc{D}^\pm(n):=\{\varepsilon=\Diag(\pm 1,...,\pm 1)\}\cong\{-1,+1\}^n$ of diagonal matrices taking their diagonal values in $\{-1,+1\}$,
    \item the group $\mf{S}^\pm(n):=\{\varepsilon P_\sigma\in\Mat(n)|(\varepsilon,\sigma)\in\mc{D}^\pm(n)\times\mf{S}(n)\}\cong\mc{D}^\pm(n)\times\mf{S}(n)$ of signed permutation matrices.
\end{enumerate}

\begin{lemma}[Characterization of inner products on symmetric matrices invariant under $\mc{D}^\pm(n)$ or $\mf{S}^\pm(n)$]\label{lemma:characterization_inner_products}
Let $\dotprod{\cdot}{\cdot}:\Sym(n)\times\Sym(n)\lto\R$ be an inner product on symmetric matrices.
\begin{enumerate}[label=(\alph*)]
    \itemsep0em
    \item \label{enum:characterization_d(n)-invariant} It is $\mc{D}^\pm(n)$-invariant if and only if there exist $\frac{n(n-1)}{2}$ positive real numbers $\alpha_{ij}=\alpha_{ji}>0$ for $i\ne j$ and a matrix $S\in\SPD(n)$ such that:
    \begin{equation}\label{eq:d(n)-invariant}
        \forall X\in\Sym(n),\dotprod{X}{X}=\sum_{i\ne j}{\alpha_{ij}X_{ij}^2}+\sum_{i,j}{S_{ij}X_{ii}X_{jj}}.
    \end{equation}
    \item \label{enum:characterization_s*(n)-invariant} It is $\mf{S}^\pm(n)$-invariant if and only if there exist $(\alpha,\beta,\gamma)\in\R^3$ with $\alpha>0$, $\gamma>\beta$ and $\gamma+(n-1)\beta>0$ such that:
    \begin{equation}\label{eq:s*(n)-invariant}
        \forall X\in\Sym(n),\dotprod{X}{X}=\gamma\,\sum_{i=1}^n{X_{ii}^2}+\alpha\,\sum_{i\ne j}{X_{ij}^2}+\beta\,\sum_{i\ne j}{X_{ii}X_{jj}}.
    \end{equation}
\end{enumerate}

\begin{proof}[Proof of Lemma \ref{lemma:characterization_inner_products}]
\begin{enumerate}[label=(\alph*)]
    \item[]
    \item We write $\dotprod{X}{X}=\sum_{i,j,k,l}{a_{ij,kl}X_{ij}X_{kl}}$ a general inner product and we use the invariance under the matrix $\varepsilon_m\in\mc{D}^\pm(n)$ with $-1$ on the $m$-th component and $1$ elsewhere, for $m\in\{1,...,n\}$. By denoting $\mc{P}\mathrm{~XOR~}\mc{Q}\in\{0,1\}$ the `exclusive or' between propositions $\mc{P}$ and $\mc{Q}$, we have $[\varepsilon_mX\varepsilon_m]_{ij}=(-1)^{(i=m)\mathrm{XOR} (j=m)}X_{ij}$. Hence, one can show that the coefficient $a_{ij,kl}$ has to be equal to $-a_{ij,kl}$, hence 0, unless $i,j,k,l$ satisfy at least one of the two following conditions:
    \begin{enumerate}[label=$\bullet$]
        \itemsep0em
        \item $\{i,j\}=\{k,l\}$,
        \item $i=j$ and $k=l$,
    \end{enumerate}
    otherwise it is possible to flip one of the two factors $X_{ij}$ or $X_{kl}$ into its opposite counterpart. We get the expression (\ref{eq:d(n)-invariant}) by denoting $\alpha_{ij}=2a_{ij,ij}$ and $S_{ij}=a_{ii,jj}=S_{ji}$. Since the quadratic form splits into two quadratic forms defined on supplementary vector spaces (off-diagonal and diagonal terms), it is positive definite if and only if these two quadratic forms are positive definite, i.e. $\alpha_{ij}>0$ for all $i\ne j$ and $S$ is positive definite. Conversely, Equation (\ref{eq:d(n)-invariant}) clearly defines $\mc{D}^\pm(n)$-invariant inner products.
    
    \item A $\mf{S}^\pm(n)$-invariant inner product on symmetric matrices is $\mc{D}^\pm(n)$-invariant so it is of the form of Equation (\ref{eq:d(n)-invariant}). Since it is invariant under permutations, we have $\alpha_{ij}=\alpha_{kl}=:\alpha$ and $S_{ij}=S_{kl}=:\beta$ for all $i\ne j$ and $k\ne l$ and $S_{ii}=S_{jj}=:\gamma$ for all $i,j$. Under these notations, Equation (\ref{eq:d(n)-invariant}) becomes Equation (\ref{eq:s*(n)-invariant}). Since $S=(\gamma-\beta)\,I_n+\beta\,\mathds{1}\mathds{1}^\top$, then $S\in\SPD(n)$ if and only if $\gamma-\beta>0$ and $\gamma-\beta+n\beta>0$ as expected. Conversely, Equation (\ref{eq:s*(n)-invariant}) clearly defines $\mf{S}^\pm(n)$-invariant inner products.
\end{enumerate}
\end{proof}

\begin{proof}[Proof of Theorem \ref{thm:characterization_o(n)-invariant_inner_products}]
An $\Orth(n)$-invariant inner product on symmetric matrices is $\mf{S}^\pm(n)$-invariant so it is of the form of Equation (\ref{eq:s*(n)-invariant}). We define the rotation matrix $R=\begin{pmatrix}R_{\pi/4} & 0\\0 & I_{n-2}\end{pmatrix}\in\Orth(n)$ with $R_{\pi/4}=\frac{\sqrt{2}}{2}\begin{pmatrix}1 & 1\\-1 & 1\end{pmatrix}\in\Orth(2)$ and we apply it to the matrix $X=\begin{pmatrix}M & Y\\Y^\top & Z\end{pmatrix}\in\Sym(n)$ with $M=\begin{pmatrix}a & b\\b & c\end{pmatrix}\in\Sym(2)$. Since $R_{\pi/4} MR_{\pi/4}^\top=\frac{1}{2}\begin{pmatrix}a+c+2b & c-a \\ c-a & a+c-2b\end{pmatrix}$, the coefficient in $b^2$ in $\dotprod{X}{X}$ is $2\alpha$ and the coefficient in $b^2$ in $\dotprod{RXR^\top}{RXR^\top}$ is $2\gamma-2\beta$. Hence by invariance, $\gamma=\alpha+\beta$ and the positivity condition becomes $\alpha>0$ and $\alpha+n\beta>0$. Conversely, Equation (\ref{eq:o(n)-invariant}) clearly defines $\Orth(n)$-invariant inner products.
\end{proof}
\end{lemma}

\section{Main $\mathrm{O}(n)$-invariant metrics on SPD matrices with new formulae}\label{sec:inventory}

The goal of this section is to describe the main $\Orth(n)$-invariant metrics on SPD matrices that can be found in the literature, namely the Euclidean (abbreviated `E', Section \ref{subsec:euclidean}), the Log-Euclidean (`LE', Section \ref{subsec:log-euclidean}), the Affine-invariant (`A', Section \ref{subsec:affine-invariant}), the Bures-Wasserstein (`BW', Section \ref{subsec:bures-wasserstein}) and the Bogoliubov-Kubo-Mori (`BKM', Section \ref{subsec:bkm}) metrics. For each metric, we give a short explanation on the way it was introduced, some useful references and a synthetic table that summarizes its fundamental Riemannian operations: squared distance, Levi-Civita connection, curvature, geodesics, logarithm map, parallel transport map (abbreviated `PT map').

Our contributions are (1) the synthesis of many results scattered in the literature especially for the Bures-Wasserstein metric, (2) the complete formula of the sectional curvature of the affine-invariant metric, (3) the new formula of the parallel transport between commuting matrices and new expressions of the Levi-Civita connection, the curvature and the parallel transport equation of the Bures-Wasserstein metric.

\subsection{$\mathrm{O}(n)$-invariant Euclidean metrics}\label{subsec:euclidean}

A Euclidean metric on SPD matrices is the pullback of an inner product $\dotprod{\cdot}{\cdot}$ on symmetric matrices by the canonical immersion $\id:\SPD(n)\lto\Sym(n)$. As we know $\Orth(n)$-invariant inner products on symmetric matrices from Theorem \ref{thm:characterization_o(n)-invariant_inner_products}, we know all the $\Orth(n)$-invariant Euclidean metrics on SPD matrices.

\begin{definition}[$\Orth(n)$-invariant Euclidean metrics on SPD matrices]
An $\Orth(n)$\textit{-invariant Euclidean metric} on SPD matrices is a Riemannian metric of the following form for all $\Sigma\in\SPD(n)$ and $X\in\Sym(n)$:
\begin{equation}
    g^{\E(\alpha,\beta)}_\Sigma(X,X)=\alpha\,\tr(X^2)+\beta\,\tr(X)^2,
\end{equation}
with $(\alpha,\beta)\in\mathbf{ST}$, i.e. $\alpha>0$ and $\beta>-\alpha/n$. Its Riemannian operations are detailed in Table \ref{tab:riemannian_operations_euclidean}.

\begin{table}[htbp]
    \centering
    \begin{tabular}{|c|c|}
    \hline
    Metric & $g_\Sigma(X,X)=\|X\|^2$ \\
    \hline
    Sq. dist. & $d(\Sigma,\Lambda)^2=\|\Lambda-\Sigma\|^2$\\
    \hline
    Levi-Civita & $\nabla_XY=\partial_XY$\\
    \hline
    Curvature & $R=0$\\
    \hline
    Geodesics & \makecell[{{p{9cm}}}]{$\gamma_{(\Sigma,X)}(t)=\Sigma+tX$ for $t\in I$ where $I$ depends on $\lambda_\mini=\min\eig(\Sigma^{-1}X)$ and $\lambda_\maxi=\max\eig(\Sigma^{-1}X)$ as follows:\\
    $\bullet\quad$ If $\lambda_\mini<0<\lambda_\maxi$, then $I=(-1/\lambda_\maxi,-1/\lambda_\mini)$.\\
    $\bullet\quad$ If $0\ls\lambda_\mini$, then $I=(-1/\lambda_\maxi,+\infty)$.\\
    $\bullet\quad$ If $\lambda_\maxi\ls0$, then $I=(-\infty,-1/\lambda_\mini)$.}\\
    \hline
    Logarithm & $\Log_\Sigma(\Lambda)=\Lambda-\Sigma$\\
    \hline
    PT map & \makecell[{{p{9cm}}}]{Does not depend on the curve:\\
    \centering
    $\Pi_{\Sigma\to\Lambda}:\fun{T_\Sigma\SPD(n)}{T_\Lambda\SPD(n)}{X}{(d_\Lambda\id)^{-1}(d_\Sigma\id(X))\equiv X}$}\\
    \hline
\end{tabular}
    \caption{Riemannian operations of $\Orth(n)$-invariant Euclidean metrics on SPD matrices}
    \label{tab:riemannian_operations_euclidean}
\end{table}
\end{definition}

\subsection{$\mathrm{O}(n)$-invariant log-Euclidean metrics}\label{subsec:log-euclidean}

A log-Euclidean metric on SPD matrices \cite{Arsigny06} is the pullback of an inner product $\dotprod{\cdot}{\cdot}$ on symmetric matrices by the symmetric matrix logarithm $\log:\SPD(n)\lto\Sym(n)$. Hence the SPD manifold endowed with the log-Euclidean metric is isometric to a Euclidean space, thus geodesically complete. From Theorem \ref{thm:characterization_o(n)-invariant_inner_products}, we know all the $\Orth(n)$-invariant log-Euclidean metrics.

\begin{definition}[$\Orth(n)$-invariant log-Euclidean metrics on SPD matrices]\label{def:log-euclidean}
An $\Orth(n)$\textit{-invariant log-Euclidean metric} on SPD matrices is a Riemannian metric of the following form for all $\Sigma\in\SPD(n)$ and $X\in\Sym(n)$:
\begin{equation}
    g^{\LE(\alpha,\beta)}_\Sigma(X,X)=\alpha\,\tr(d_\Sigma\log(X)^2)+\beta\,\tr(\Sigma^{-1}X)^2,
\end{equation}
with $(\alpha,\beta)\in\mathbf{ST}$, i.e. $\alpha>0$ and $\beta>-\alpha/n$. Moreover, this metric is the pullback of the Frobenius log-Euclidean metric ($(\alpha,\beta)=(1,0)$) by the isometry $f_{p,q}:\Sigma\in\SPD(n)\lmto\exp (F_{p,q}(\log\Sigma))=\det(\Sigma)^{\frac{p-q}{n}}\Sigma^q\in\SPD(n)$ with $p=\sqrt{\alpha+n\beta}$ and $q=\sqrt\alpha$, where $F_{p,q}$ was defined in Theorem \ref{thm:characterization_o(n)-invariant_inner_products}. It is geodesically complete. Its Riemannian operations are detailed in Table \ref{tab:riemannian_operations_log-euclidean}.

\begin{table}[htbp]
    \centering
\begin{tabular}{|c|c|}
    \hline
    Metric & $g_\Sigma(X,X)=\|d_\Sigma\log(V)\|^2$ \\
    \hline
    Sq. dist. & $d(\Sigma,\Lambda)^2=\|\log\Lambda-\log\Sigma\|^2$\\
    \hline
    Levi-Civita & $\nabla_XY=\partial_X^{\log{}}Y$\\
    \hline
    Curvature & $R=0$\\
    \hline
    Geodesics & $\forall t\in\R,\gamma_{(\Sigma,X)}(t)=\exp(\log(\Sigma)+t\,d_\Sigma\log(X))$\\
    \hline
    Logarithm & $\Log_\Sigma(\Lambda)=(d_\Sigma\log)^{-1}(\log\Lambda-\log\Sigma)$\\
    \hline
    PT map & \makecell[{{p{9cm}}}]{Does not depend on the curve:\\
    \centering
    $\Pi_{\Sigma\to\Lambda}:\fun{T_\Sigma\SPD(n)}{T_\Lambda\SPD(n)}{X}{(d_\Lambda\log)^{-1}(d_\Sigma\log(X))}$}\\
    \hline
\end{tabular}
    \caption{Riemannian operations of $\Orth(n)$-invariant log-Euclidean metrics on SPD matrices}
    \label{tab:riemannian_operations_log-euclidean}
\end{table}
\end{definition}

\subsection{Affine-invariant metrics}\label{subsec:affine-invariant}

Affine-invariant metrics were introduced in many different ways. Siegel introduced a metric on the half space $\mc{S}=\{X+i\Sigma|\,X\in\Sym(n),\Sigma\in\SPD(n)\}$ which is invariant under automorphisms \cite{Siegel43}. The restriction of this metric to SPD matrices by the immersion $\Sigma\in\SPD(n)\hookrightarrow i\Sigma\in\mc{S}$ is $g_\Sigma(X,Y)=\tr(\Sigma^{-1}X\Sigma^{-1}Y)$.

Rao considered the Fisher information of a family of densities as a Riemannian metric on the space of parameters \cite{Rao45} and Skovgaard detailed all the properties of the Fisher-Rao metric of the family of multivariate Gaussian densities \cite{Skovgaard84}. By restriction to the family of \textit{centered} multivariate Gaussian densities, we get the same metric as Siegel's scaled by a factor $1/2$, namely $g_\Sigma(X,Y)=\frac{1}{2}\tr(\Sigma^{-1}X\Sigma^{-1}Y)$. In addition, Amari stated that the canonical immersion $\id:\Sigma\in\SPD(n)\lmto\Sigma\in\Sym(n)$ and the inversion $\inv:\Sigma\in\SPD(n)\lmto\Sigma^{-1}\in\Sym(n)$ give two dual coordinate systems with respect to this metric \cite{Amari00}.

Between 2005 and 2007, this metric was used in many computational methods for Diffusion Tensor Imaging \cite{Pennec06,Lenglet06-JMIV,Fletcher07,Moakher05,Batchelor05}, in functional MRI \cite{Varoquaux10} and in Brain-Computer Interfaces \cite{Barachant13}. In particular, Pennec's approach \cite{Pennec08} consisted in finding \textit{all} the metrics on $\SPD(n)$ that are invariant under the congruence action $\Sigma\in\SPD(n)\lmto A\Sigma A^\top\in\SPD(n)$ for $A\in\GL(n)$, which corresponds to the affine action $X\in\R^n\lmto AX+B\in\R^n$ on the empirical covariance matrix $\Sigma=\frac{1}{n}(X-\bar{X})(X-\bar{X})^\top$. Thus, affine-invariant metrics are characterized by an $\Orth(n)$-invariant inner product on the tangent space at $I_n$, that is on symmetric matrices. Hence we know all the affine-invariant metrics from Theorem \ref{thm:characterization_o(n)-invariant_inner_products}.

\begin{definition}[Affine-invariant metrics on SPD matrices]
An \textit{affine-invariant metric} on SPD matrices is a Riemannian metric of the following form for all $\Sigma\in\SPD(n)$ and $X\in\Sym(n)$:
\begin{equation}
    g^{\A(\alpha,\beta)}_\Sigma(X,X)=\alpha\,\tr((\Sigma^{-1}X)^2)+\beta\,\tr(\Sigma^{-1}X)^2,
\end{equation}
with $(\alpha,\beta)\in\mathbf{ST}$, i.e. $\alpha>0$ and $\beta>-\alpha/n$. The Fisher-Rao metric often refers to the affine-invariant metric with $(\alpha,\beta)=(1/2,0)$. Moreover, given $\alpha>0$, this metric is the pullback of the affine-invariant metric with $\beta=0$ by the isometry $f_{p,1}:\Sigma\in\SPD(n)\lmto\det(\Sigma)^{\frac{p-1}{n}}\Sigma\in\SPD(n)$ with $p=\sqrt{\frac{\alpha+n\beta}{\alpha}}$.
\end{definition}

The following proposition details the characteristics of homogeneity and symmetry of these Riemannian metrics. The Riemannian operations, essentially due to Skovgaard \cite{Skovgaard84}, are detailed in Table \ref{tab:riemannian_operations_affine-invariant}. The second term of the sectional curvature is part of our contributions as it seems to be forgotten in \cite{Skovgaard84}.

\begin{proposition}[Riemannian symmetric structure of the affine-invariant metric]\label{prop:Riemannian_symmetric_characteristics_of_the_Affine-invariant_metric}
The Riemannian manifold $(\SPD(n),g^{\A(\alpha,\beta)})$ is a Riemannian symmetric space, hence it is geodesically complete. The underlying homogeneous space is $\GL^+(n)/\SO(n)$ and $g^{\A(\alpha,\beta)}$ is a quotient metric obtained by the submersion $\pi:A\in\GL^+(n)\lmto AA^\top\in\SPD(n)$ from the left-invariant metric $G_A(M,M)=4\alpha\,\tr(A^{-1}M(A^{-1}M)^\top)+4\beta\,\tr(A^{-1}M)^2$ for $A\in\GL^+(n)$ and $M\in T_A\GL^+(n)$. The symmetries are $s_\Sigma:\Lambda\in\SPD(n)\lmto\Sigma\Lambda^{-1}\Sigma\in\SPD(n)$.

\begin{table}[htbp]
    \centering
\begin{tabular}{|c|c|}
    \hline
    Metric & $g_\Sigma(X,X)=\alpha\|\Sigma^{-1}X\|^2+\beta\,\tr(\Sigma^{-1}X)^2$ \\
    \hline
    Sq. dist. & $d(\Sigma,\Lambda)^2=\alpha\|\log(\Sigma^{-1/2}\Lambda\Sigma^{-1/2})\|^2+\beta\,\log(\det(\Sigma^{-1}\Lambda))^2$\\
    \hline
    Levi-Civita & \makecell[{{p{9cm}}}]{\centering $(\nabla_XY)_{|\Sigma}=(\partial_XY)_{|\Sigma}-\frac{1}{2}(X\Sigma^{-1}Y+Y\Sigma^{-1}X)$}\\
    \hline
    Curvature &
    \makecell[{{p{9cm}}}]{
        The sectional curvature is non-positive and bounded. More precisely, the Riemann and sectional curvatures are:\\
        $\,\,R_\Sigma(X,Y,Z,T)  =\frac{\alpha}{2}(X\Sigma^{-1}Y\Sigma^{-1}(Z\Sigma^{-1}T-T\Sigma^{-1}Z)\Sigma^{-1})$\\
        $\kappa(\Sigma^{1/2}E_{ii}^\beta\Sigma^{1/2},\Sigma^{1/2}E_{ij}^\beta\Sigma^{1/2})=-1/4\alpha$ for $i\ne j$\\
        $\kappa(\Sigma^{1/2}E_{ij}^\beta\Sigma^{1/2},\Sigma^{1/2}E_{ik}^\beta\Sigma^{1/2})=-1/8\alpha$ for $i\ne j\ne k\ne i$\\
        where $E_{ij}^\beta=E_{ij}-\frac{1-p}{np}\delta_{ij}I_n$. Other terms are null.
    }\\
    \hline
    Geodesics & $\forall t\in\R,\gamma_{(\Sigma,X)}(t)=\Sigma^{1/2}\exp(t\,\Sigma^{-1/2}X\Sigma^{-1/2})\Sigma^{1/2}$\\
    \hline
    Logarithm & $\Log_\Sigma(\Lambda)=\Sigma^{1/2}\log(\Sigma^{-1/2}\Lambda\Sigma^{-1/2})\Sigma^{1/2}$\\
    \hline
    PT map & \makecell[{{p{9cm}}}]{Depends on the curve. Along a geodesic:\\
    \centering
    $\Pi_{\Sigma\to\Lambda}:\fun{T_\Sigma\SPD(n)}{T_\Lambda\SPD(n)}{X}{(\Lambda\Sigma^{-1})^{1/2}X(\Sigma^{-1}\Lambda)^{1/2}}$}\\
    \hline
\end{tabular}
    \caption{Riemannian operations of Affine-invariant metrics on SPD matrices}
    \label{tab:riemannian_operations_affine-invariant}
\end{table}
\end{proposition}

\begin{proof}[Proof of sectional curvature in Table \ref{tab:riemannian_operations_affine-invariant}]
Firstly, we compute the sectional curvature of the affine-invariant metrics for $\beta=0$ at $\Sigma\in\SPD(n)$ in the orthonormal basis $(\Sigma^{1/2}E_{ij}\Sigma^{1/2})_{1\ls i\ls j\ls n}$, with $E_{ii},E_{ij}$ for $i\ne j$ defined by $E_{ii}(k,l)=\delta_{ik}\delta_{il}$ and $E_{ij}(k,l)=\frac{\delta_{ik}\delta_{jl}+\delta_{il}\delta_{jk}}{\sqrt{2}}$. As $\kappa(X,Y)=\frac{R(X,Y,X,Y)}{\|X\|^2\|Y\|^2-\dotprod{X}{Y}^2}$, we have $\kappa(\Sigma^{1/2}E_{ij}\Sigma^{1/2},\Sigma^{1/2}E_{kl}\Sigma^{1/2})=\frac{1}{2\alpha}\tr((E_{ij}E_{kl})^2-(E_{ij}E_{kl})(E_{ij}E_{kl})^\top)$ so we only need to compute a few expressions. In the following equalities, when an elementary matrix $E$ has two different indexes, they are assumed to be distinct:\\
\begin{tabular}{p{5cm}l}
    $\bullet\,\,\,E_{ii}E_{jj} = \delta_{ij}C_{ij}$ & hence $\|E_{ii}E_{jj}\|^2 = \delta_{ij}$, \\
    $\bullet\,\,\,E_{ii}E_{jk} = \frac{1}{\sqrt{2}}(\delta_{ij}C_{ik}+\delta_{ik}C_{ij})$ & hence $\|E_{ii}E_{jk}\|^2 = \frac{1}{2}(\delta_{ij}+\delta_{ik})$,\\
    \multicolumn{2}{l}{$\bullet\,\,\,E_{ij}E_{kl} = \frac{1}{2}(\delta_{jk}C_{il}+\delta_{ik}C_{jl}+\delta_{jl}C_{ik}+\delta_{il}C_{jk})$}\\
    & hence $\|E_{ij}E_{kl}\|^2 = \frac{1}{4}(\delta_{ik}+\delta_{il}+\delta_{jk}+\delta_{jl})$, \\
\end{tabular}\\
\begin{tabular}{p{5cm}l}
    $\bullet\,\,\,(E_{ii}E_{jj})^2 = \delta_{ij}C_{ij}$ & hence $\tr((E_{ii}E_{jj})^2) = \delta_{ij}$,\\
    $\bullet\,\,\,(E_{ii}E_{jk})^2 = 0$ & hence $\tr((E_{ii}E_{jk})^2) = 0$,\\
    \multicolumn{2}{l}{$\bullet\,\,\,(E_{ij}E_{kl})^2 = \frac{1}{4}(\delta_{jk}\delta_{il}(C_{il}+C_{jk})+\delta_{jl}\delta_{ik}(C_{ik}+C_{jl}))$,} \\
    & hence $\tr((E_{ij}E_{kl})^2) = \frac{1}{2}(\delta_{jk}\delta_{il}+\delta_{jl}\delta_{ik})$.
\end{tabular}\\
\begin{tabular}{l}
    $\bullet\,\,\,\kappa(E_{ii},E_{jj})=0$, \\
    $\bullet\,\,\,\kappa(E_{ii},E_{jk})=-\frac{1}{4\alpha}(\delta_{ij}+\delta_{ik})$,\\
    $\bullet\,\,\,\kappa(E_{ij},E_{kl})=-\frac{1}{8\alpha}((\delta_{ik}-\delta_{jl})^2+(\delta_{il}-\delta_{jk})^2)$.
\end{tabular}\\
Hence the non null terms are $\kappa(E_{ii},E_{ij})=-\frac{1}{4\alpha}$ and $\kappa(E_{ij},E_{ik})=-\frac{1}{8\alpha}$.\\
Secondly, for $\beta\ne 0$, we use the isometry $f_{p,1}$: the values are the same if we replace $\Sigma^{1/2}E_{ij}\Sigma^{1/2}$ by $(d_\Sigma f_{p,1})^{-1}(f_{p,1}(\Sigma)^{1/2}E_{ij}f_{p,1}(\Sigma)^{1/2})=\Sigma^{1/2}E_{ij}^\beta\Sigma^{1/2}$.
\end{proof}

Another metric that also provides a Riemannian symmetric structure on $\SPD(n)$ was used in \cite{Su12,Zhang18}. It was introduced directly by the quotient structure detailed in Proposition \ref{prop:Riemannian_symmetric_characteristics_of_the_Affine-invariant_metric} but with the submersion $\sqrt\pi:A\in\GL^+(n)\lmto\sqrt{AA^\top}\in\SPD(n)$ based on the polar decomposition of $A$ (and without the coefficient $4$). We called it the Polar-Affine metric in \cite{Thanwerdas19-a}. It is $\GL(n)$-invariant with respect to the action $(A,\Sigma)\in\GL(n)\times\SPD(n)\lmto\sqrt{A\Sigma^2A^\top}\in\SPD(n)$. Hence it is $\Orth(n)$-invariant in the usual sense. It is the pullback metric of the affine-invariant metric via the square diffeomorphism $\pow_2:\Sigma\lmto\Sigma^2$ \cite{Thanwerdas19-a}.

\subsection{Bures-Wasserstein metric}\label{subsec:bures-wasserstein}

The $L^2$-Wasserstein distance between multivariate centered Gaussian distributions is given by $d(\Sigma,\Lambda)^2=\tr\Sigma+\tr\Lambda-2\tr((\Sigma\Lambda)^{1/2})$. It corresponds to the Procrustes distance between square-root matrices, namely $d(\Sigma,\Lambda)^2=\inf_{U\in\Orth(n)}\|\Sigma^{1/2}-\Lambda^{1/2}U\|_\Frob^2$. The second order approximation of this squared distance defines a Riemannian metric called the Bures metric (or the Helstrom metric) in quantum physics. All these viewpoints are explained in details with modern notations in \cite{Bhatia19}. In particular, the expression of the Riemannian metric is derived in \cite{Bhatia19} and we take it as a definition.

\begin{definition}[Bures-Wasserstein metric]
The \textit{Bures-Wasserstein metric} is the Riemannian metric associated to the Bures-Wasserstein distance. It is $\Orth(n)$-invariant and given an eigenvalue decomposition $\Sigma=PDP^\top\in\SPD(n)$ with $P\in\Orth(n)$ and $D=\diag(d_1,...,d_n)$ and $X=PX'P^\top$, its expression is:
\begin{equation}
    g^{\BW}_\Sigma(X,X)=g^{\BW}_D(X',X')=\frac{1}{2}\sum_{i,j}{\frac{1}{d_i+d_j}X_{ij}'^2}.
\end{equation}
\end{definition}

The Bures-Wasserstein metric can also be expressed by means of the linear map $\mc{S}_\Sigma:\Sym(n)\lto\Sym(n)$ implicitly defined by the Sylvester equation $X=\Sigma\mc{S}_\Sigma(X)+\mc{S}_\Sigma(X)\Sigma$ for $X\in\Sym(n)$. More explicitly with the previous notations, we have $\mc{S}_\Sigma(X)=P\left[\frac{X'_{ij}}{d_i+d_j}\right]_{i,j}P^\top$. Then we have $g_\Sigma^{\BW}(X,Y)=\frac{1}{2}\tr(X\mc{S}_\Sigma(Y))=\tr(\mc{S}_\Sigma(X)\Sigma\mc{S}_\Sigma(Y))$, where $X,Y\in T_\Sigma\SPD(n)$ are canonically identified with $d_\Sigma\id(X),d_\Sigma\id(Y)\in\Sym(n)$, as explained in the introduction. This is a common expression in recent papers \cite{Malago18,Oostrum20}. However, in \cite{Takatsu11} which is a reference paper on the Bures-Wasserstein metric, Takatsu gives the expression $g_\Sigma(X,Y)=\tr(\mathbf{X}\Sigma\mathbf{Y})$. The trick comes from the identification $\mc{S}_\Sigma(X)\equiv\mathbf{X}\in\Sym(n)$ that differs from the canonical one $d_\Sigma\id(X)\equiv\mathbf{X}\in\Sym(n)$. As this could be confusing when the formula is written without this precision (and without bold letters), we adopt the same formalism as \cite{Malago18,Bhatia19,Oostrum20}.

\begin{table}[htbp]
    \centering
    \begin{tabular}{|c|c|}
        \hline
        Bundle & $\GL(n)$ \\
        \hline
        Group action & $\rho:(A,U)\in\GL(n)\times\Orth(n)\lmto AU\in\GL(n)$ \\
        \hline
        Submersion & $\pi:A\in\GL(n)\lmto \Sigma:=AA^\top\in\SPD(n)$ \\
        \hline
        Vertical space & $\mc{V}_A=\ker{d_A\pi}=\Skew(n)\,A^{-\top}$ \\
        \hline
        Bundle metric & $G_A(M,M)=\tr(MM^\top)$ \\
        \hline
        Hor. space & $\mc{H}_A=\mc{V}_A^{\perp_{G}}=\Sym(n)A$ \\
        \hline
        Hor. isometry & $(d_A\pi)_{|\mc{H}_A}:\fun{\mc{H}_A=\Sym(n)A}{T_\Sigma\SPD(n)}{X^h=X^0A}{X=\Sigma X^0+X^0\Sigma}$ \\
        \hline
        Sym. lift $X^0$ & $\mc{S}_\Sigma:\fun{T_\Sigma\SPD(n)}{\mc{H}_{I_n}=\Sym(n)}{X}{X^0=P{X^0}'P^\top\mathrm{with}\,{X^0_{ij}}'=\frac{X'_{ij}}{d_i+d_j}}$ \\
        \hline
        Hor. lift $X^h$ & $X\in T_\Sigma\SPD(n)\lmto X^h=X^0A\in\mc{H}_A$ \\
        \hline
    \end{tabular}
    \caption{Quotient structure of the Bures-Wasserstein metric}
    \label{tab:quotient_bures-wasserstein}
\end{table}

We recall the quotient structure of the Bures-Wasserstein metric \cite{Bhatia19} in Table \ref{tab:quotient_bures-wasserstein}. The Riemannian operations are detailed in Table \ref{tab:riemannian_operations_bures-wasserstein}. Let us precise what was known and what is new in Table \ref{tab:riemannian_operations_bures-wasserstein}. 

The proofs of the formulae of the distance and the logarithm can be found in \cite{Bhatia19}. The Levi-Civita connection and the exponential map were computed in \cite{Malago18}. We computed the Levi-Civita connection independently using a more geometric proof provided in Appendix A. We get a simpler formula.

Takatsu computed the curvature in \cite{Takatsu10} in a basis of vectors and gave a general formula in \cite{Takatsu11}. However, we argued above that the notations of \cite{Takatsu11} could be confusing because of the chosen identification. Moreover, the expression of the curvature given there is a bit implicit since it is $R_\Sigma(X,Y,X,Y)=\frac{3}{4}\tr((|Y,X]-S)\Sigma([Y,X]-S)^\top)$ where $S=\mc{S}_\Sigma([X,Y]\Sigma+\Sigma[Y,X])\in\Sym(n)$. For this reason, we prove in Appendix A the compact and explicit formula provided in Table \ref{tab:riemannian_operations_bures-wasserstein} using the same method, equations of submersions \cite{ONeill66}.

Finally, the geodesic parallel transport between commuting SPD matrices is new. We provide a new formulation of the equation of the parallel transport between any two SPD matrices in the following proposition. The proofs are given in Appendix A.

\begin{table}[htbp]
    \centering
\begin{tabular}{|c|c|}
    \hline
    Metric & \makecell[{{p{9cm}}}]{\centering $g_\Sigma(X,X)=g_{\Sigma^{1/2}}(X^h,X^h)=\frac{1}{2}\sum_{i,j}{\frac{1}{d_i+d_j}X_{ij}'^2}$}\\
    \hline
    Sq. dist. & $d(\Sigma,\Lambda)^2=\tr\Sigma+\tr\Lambda-2\tr((\Sigma\Lambda)^{1/2})$\\
    \hline
    Levi-Civita & $(\nabla_XY)_{|\Sigma}=(\partial_XY)_{|\Sigma}-(X^0\Sigma Y^0+Y^0\Sigma X^0)$\\
    \hline
    Curvature &
    \makecell[{{p{9cm}}}]{The sectional curvature is non-negative.\\ More precisely $R_\Sigma(X,Y,X,Y)=\frac{3}{2}\sum_{i,j}{\frac{d_id_j}{d_i+d_j}\left[{X^0}',{Y^0}'\right]_{ij}^2}$ where $[V,W]=VW-WV$ is the Lie bracket of matrices.}\\
    \hline
    Geodesics & \makecell[{{p{9cm}}}]{$\gamma_{(\Sigma,X)}(t)=\Sigma+tX+t^2X^0\Sigma X^0$
    for $t\in I$ where $I$ depends on $\lambda_\maxi=\max\eig(X^0)$ and $\lambda_\mini=\min\eig(X^0)$ as follows:\\
    $\bullet\quad$ If $\lambda_\mini<0<\lambda_\maxi$, then $I=(-1/\lambda_\maxi,-1/\lambda_\mini)$.\\
    $\bullet\quad$ If $0\ls\lambda_\mini$, then $I=(-1/\lambda_\maxi,+\infty)$.\\
    $\bullet\quad$ If $\lambda_\maxi\ls0$, then $I=(-\infty,-1/\lambda_\mini)$.}\\
    \hline
    Logarithm & $\Log_\Sigma(\Lambda)=(\Sigma\Lambda)^{1/2}+(\Lambda\Sigma)^{1/2}-2\Sigma$\\
    \hline
    PT map & \makecell[{{p{9cm}}}]{Depends on the curve. Along a geodesic between commuting matrices $\Sigma=PDP^\top$ and $\Lambda=P\Delta P^\top$:\\
    {\centering
    $\Pi_{\Sigma\to\Lambda}:\fun{T_\Sigma\SPD(n)}{T_\Lambda\SPD(n)}{X}{P\left[\sqrt{\frac{\delta_i+\delta_j}{d_i+d_j}}[P^\top XP]_{ij}\right]P^\top}$}}\\
    \hline
\end{tabular}
    \caption{Riemannian operations of the Bures-Wasserstein metric on SPD matrices}
    \label{tab:riemannian_operations_bures-wasserstein}
\end{table}

\begin{proposition}[Parallel transport equation of Bures-Wasserstein metric]
Let $\gamma(t)$ the geodesic between $\gamma(0)=\Sigma$ and $\gamma(1)=\Lambda$, and a vector $X\in T_\Sigma\SPD(n)$. We denote $\gamma^h(t)=(1-t)\Sigma^{1/2}+t\Sigma^{-1/2}(\Sigma^{1/2}\Lambda\Sigma^{1/2})^{1/2}$ the horizontal lift of the geodesic $\gamma$. The two following statements are equivalent.
\begin{enumerate}[label=(\roman*)]
    \itemsep0em
    \item The vector field $X(t)$ defined along $\gamma(t)$ is the parallel transport of $X$.
    \item $X(t)=\gamma(t)X^0(t)+X^0(t)\gamma(t)$ where $X^0(t)$ is a curve in $\Sym(n)$ satisfying the following ODE:
\begin{equation}
    \gamma(t)\dot{X}^0(t)+\dot{X}^0(t)\gamma(t)+\gamma^h(t)\dot\gamma^{h\top} X^0(t)+X^0(t)\dot\gamma^h\gamma^h(t)^\top=0.
\end{equation}
\end{enumerate}
\end{proposition}

\subsection{Bogoliubov-Kubo-Mori metric}\label{subsec:bkm}

The Bogoliubov-Kubo-Mori metric is a Riemannian metric used in quantum physics \cite{Petz93}, given by $g_\Sigma^{\BKM}(X,X)=\tr(\int_0^\infty{(\Sigma+t\,I_n)^{-1}X(\Sigma+t\,I_n)^{-1}X dt})$. It can be seen as the integration of the affine-invariant metric on a half-line included in the SPD cone. It can be rewritten thanks to the differential of the logarithm and we take this other expression as a definition.

\begin{definition}[Bogoliubov-Kubo-Mori (BKM) metric]
The \textit{Bogoliubov-Kubo-Mori metric} is the $\Orth(n)$-invariant Riemannian metric defined for $\Sigma\in\SPD(n)$ and $X\in T_\Sigma\SPD(n)$ by:
\begin{equation}
    g^\BKM_\Sigma(X,X)=\tr(X\,d_\Sigma\log(X)).
\end{equation}
\end{definition}

Important functions related to this metric are defined by \cite{Michor00} to get simple expressions of the Levi-Civita connection and the curvature. Given $\Sigma=PDP^\top\in\SPD(n)$, they define $m_{ij}=\int_0^{\infty}{(d_i+t)^{-1}(d_j+t)^{-1}dt}$ which is symmetric in $(i,j)$ and $m_{ijk}=\int_0^{\infty}{(d_i+t)^{-1}(d_j+t)^{-1}(d_k+t)^{-1}dt}$ which is symmetric in $(i,j,k)$. They also denote $g_\Sigma(X)=d_\Sigma\log(X)$ whose expression is $g_\Sigma(X)=P\,g_D(X')\,P^\top$ and $[g_D(X')]_{ij}=m_{ij}X'_{ij}$ where $X'=P^\top XP$. This $g_\Sigma$ is defined so that $g_\Sigma(X,Y)=\tr(X\,g_\Sigma(Y))$. By differentiating this equality and using the definition of the BKM metric, they get the differential of $\Sigma\lmto g_\Sigma$:
\begin{align*}
    d_\Sigma g(PF_{ij}P^\top)&(PF_{kl}P^\top) =d_D g(F_{ij})(F_{kl})\\
    &=-\frac{1}{2}(\delta_{jk}m_{ilj}F_{il}+\delta_{jl}m_{ikj}F_{ik}+\delta_{il}m_{jki}F_{jk}+\delta_{ik}m_{jli}F_{jl}),
\end{align*}
or more compactly $[d_\Sigma g(PXP^\top)(PXP^\top)]_{ij}=-2\sum_{k=1}^n{m_{ijk}X_{ik}X_{jk}}$. The Levi-Civita connection and the curvature can be expressed in closed forms by means of $g$ and $dg$, as shown in Table \ref{tab:riemannian_operations_bkm}. Note that the sign of the sectional curvature is not known. The distance, exponential, logarithm and parallel transport maps are not known either.

\begin{table}[htbp]
    \centering
    \begin{tabular}{|c|c|}
    \hline
    Metric & \makecell[{{p{7.5cm}}}]{\centering $g_\Sigma(X,X)=\tr(X\,d_\Sigma\log(X))$}\\
    \hline
    Levi-Civita & $(\nabla_XY)_{|\Sigma}=(\partial_XY)_{|\Sigma}+\frac{1}{2}g_\Sigma^{-1}(d_\Sigma g(X)(Y))$\\
    \hline
    Curvature &
    \makecell[{{p{7.5cm}}}]{$ R_\Sigma(X,Y)Z=-\frac{1}{4}g_\Sigma^{-1}(d_\Sigma g(X)(g_\Sigma^{-1}(d_\Sigma g(Y)(Z))))$\\
    $\quad\quad\quad\quad\quad\quad+\frac{1}{4}g_\Sigma^{-1}(d_\Sigma g(Y)(g_\Sigma^{-1}(d_\Sigma g(X)(Z))))$}\\
    \hline
\end{tabular}
    \caption{Riemannian operations of the BKM metric on SPD matrices}
    \label{tab:riemannian_operations_bkm}
\end{table}

In this section, we reviewed five of the mainly used $\Orth(n)$-invariant Riemannian metrics and we contributed new formulae. We also highlighted that the $\Orth(n)$-invariant Euclidean, the $\Orth(n)$-invariant log-Euclidean and the affine-invariant metrics are actually two-parameter families of Riemannian metrics indexed by $(\alpha,\beta)\in\mathbf{ST}$ while this extra term weighted by the trace factor $\beta$ is never defined in the literature for the Bures-Wasserstein and the Bogoliubov-Kubo-Mori metrics. Actually, there does not seem to exist a natural way of extending them with a trace term. Indeed, under the Bures-Wasserstein metric, there is a choice of an $\Orth(n)$-right-invariant inner product on $\GL(n)$ but they differ from $\Orth(n)$-invariant inner products on symmetric matrices given in Theorem \ref{thm:characterization_o(n)-invariant_inner_products}. Indeed, any inner product on $\GL(n)$ of the form $\dotprod{X}{X}=\tr(X^\top SX)$ with $S\in\SPD(n)$ is $\Orth(n)$-right-invariant. As for the BKM metric, we could change the inner product in the integral but after computation, we would obtain this metric: $\alpha\,g_\Sigma^\BKM(X,X)+\beta\sum_{i,j}{\log^{[1]}(d_i,d_j)X_{ii}'X_{jj}'}$. The fact that we cannot separate the indices $i$ and $j$ in the trace term differs from the previous situations.

In the next section, we recall the definition of the class of kernel metrics \cite{Hiai09,Hiai12} and a selection of its key properties. Since this class of Riemannian metrics contains all the previously introduced metrics without trace term, we show that this is the right framework to define the trace term extension. We show that this new class of extended kernel metrics still satisfies the key results on kernel metrics we selected. We also prove another property of these two classes: the stability under the cometric.

\section{The interpolating class of kernel metrics: new observations}\label{sec:kernel}

Kernel metrics were introduced by Hiai and Petz in 2009 \cite{Hiai09}. It is a family of $\Orth(n)$-invariant metrics indexed by smooth bivariate functions $\phi:(0,\infty)^2\lto(0,\infty)$ called kernels. It has several key properties and it encompasses all the $\Orth(n)$-invariant metrics introduced in Section \ref{sec:inventory} without trace factor ($\beta=0$). After recalling these key results (Section \ref{subsec:kernel_metrics}), we provide new observations on kernel metrics (Section \ref{subsec:new_observations}), especially the trace term extension and the stability under the cometric.

\subsection{The general class of kernel metrics}\label{subsec:kernel_metrics}

\begin{definition}[Kernel metrics, mean kernel metrics]\label{def:kernel_metrics}\cite{Hiai09}
A \textit{kernel metric} is an $\Orth(n)$-invariant metric for which there is a smooth bivariate map $\phi:(0,\infty)^2\lto(0,\infty)$ such that $g_\Sigma(X,X)=g_D(X',X')=\sum_{i,j}{\frac{1}{\phi(d_i,d_j)}X_{ij}'^2}$, where $\Sigma=PDP^\top$ with $P\in\Orth(n)$ and $D=\Diag(d_1,...,d_n)$, and $X=PX'P^\top$.

A \textit{mean kernel metric} is a kernel metric characterized by a bivariate map $\phi$ of the form $\phi(x,y)=a\,m(x,y)^\theta$ where $a>0$ is a positive coefficient, $\theta\in\R$ is a homogeneity power and $m:(0,\infty)^2\lto(0,\infty)$ is a symmetric homogeneous mean, that is:
\begin{enumerate}
    \itemsep0em
    \item symmetric, i.e. $m(x,y)=m(y,x)$ for all $x,y>0$,
    \item homogeneous, i.e. $m(\lambda x,\lambda y)=\lambda\,m(x,y)$ for all $\lambda,x,y>0$,
    \item non-decreasing in both variables,
    \item $\min(x,y)\ls m(x,y)\ls\max(x,y)$ for all $x,y>0$. It implies $m(x,x)=x$.
\end{enumerate}
\end{definition}

As the goal of this paper is to extend the class of kernel metrics, we selected from \cite{Hiai09,Hiai12} the results that we found simple and powerful to be able to generalize them later on. It would be interesting to study other properties such as monotonicity and comparison properties but it is beyond the scope of this paper. Our selection of results is in Proposition \ref{prop:main_properties_kernel_metrics}.

\begin{proposition}[Key results on kernel metrics]\cite{Hiai09}\label{prop:main_properties_kernel_metrics}
\begin{enumerate}
    \item (Generality) The Euclidean, log-Euclidean and affine-invariant metrics without trace term ($\beta=0$), the polar-affine, the Bures-Wasserstein and the Bogoliubov-Kubo-Mori metrics are mean kernel metrics. The kernels and the names of the corresponding means are given in Table \ref{tab:mean_kernel_metrics}.
    
\begin{table}[htbp]
    \centering
    \begin{tabular}{|c|c|c|c|c|}
    \hline
    Metric & $\phi(x,y)$ & Mean $m$ & $\theta$\\
    \hline
    Euclidean & $1$ & Any mean & $0$ \\
    Log-Euclidean & $(\frac{x-y}{\log(x)-\log(y)})^2$ & Logarithmic mean & $2$\\
    Affine-invariant & $xy$ & Geometric mean & $2$\\
    Polar-affine & $(\frac{2xy}{x+y})^2$ & Harmonic mean & $2$\\
    Bures-Wasserstein & $4\,\frac{x+y}{2}$ & Arithmetic mean & $1$ \\
    BKM & $\frac{x-y}{\log(x)-\log(y)}$ & Logarithmic mean & $1$\\
    \hline
    \end{tabular}
    \caption{Bivariate functions of all the $\Orth(n)$-invariant metrics of Section \ref{sec:inventory}.}
    \label{tab:mean_kernel_metrics}
\end{table}

    \item (Stability) The class of kernel metrics is stable under univariate diffeomorphisms. More precisely, if $g$ is a kernel metric with kernel function $\phi$ and if $f$ is a univariate diffeomorphism (defined in Section \ref{subsec:univariate_maps}), then the pullback metric $f^*g$ is a kernel metric with bivariate function $(x,y)\lmto\frac{\phi(f(x),f(y))}{f^{[1]}(x,y)^2}$. Note that the class of mean kernel metrics is not stable under univariate diffeomorphisms because of the non-decreasing property required for mean kernel metrics.
    
    \item (Completeness) A mean kernel metric with homogeneity power $\theta$ is geodesically complete if and only if $\theta=2$. Therefore this result provides a sufficient condition for kernel metrics to be geodesically complete.
\end{enumerate}
\end{proposition}

Another property that we left for a different reason is the attractivity of the Log-Euclidean metric, i.e. the fact that the Log-Euclidean metric is the limit when $p$ tends to 0 of the pullback of a kernel metric by a power diffeomorphism $\pow_p:\Sigma\in\SPD(n)\lmto\Sigma^p\in\SPD(n)$, scaled by $\frac{1}{p^2}$. However, it is not specific to kernel metrics since this is the case for any metric $g$.

\subsection{New observations on kernel metrics}\label{subsec:new_observations}

\subsubsection{Kernel metrics form a cone}\label{subsubsec:cone}

The class of kernel metrics is a sub-cone of the cone of Riemannian metrics on the SPD manifold. Indeed, it is stable by positive scaling and it is convex because if $g,g'$ are kernel metrics associated to $\phi,\phi'$, then $(1-t)g+tg'$ is a kernel metric associated to $\phi\phi'/((1-t)\phi'+t\phi)>0$ for $t\in[0,1]$.

\subsubsection{Cometric stability of the class of kernel metrics}\label{subsubsec:cometric_stability}

A Riemannian metric $g:T\mc{M}\times T\mc{M}\lto\R$ on a manifold $\mc{M}$ defines a cometric $g^*:T^*\mc{M}\times T^*\mc{M}\lto\R$ defined for all covectors $\omega,\omega'\in T^*\mc{M}$ by $g^*(\omega,\omega')=\omega(x')$ where $x'\in T\mc{M}$ is the unique vector such that for all vectors $x\in T\mc{M}$, $g(x,x')=\omega'(x)$ (Riesz's theorem).

On the manifold of SPD matrices $\mc{M}=\SPD(n)$, we have a canonical identification of $T_\Sigma\mc{M}$ with $\Sym(n)$ given by $d_\Sigma\id$. Hence by duality, we also have a canonical identification between $T^*_\Sigma\mc{M}$ to $\Sym(n)^*$. So to identify $T_\Sigma\mc{M}$ with $T^*_\Sigma\mc{M}$, we only need an identification between $\Sym(n)$ and $\Sym(n)^*$. This is provided by the Frobenius inner product. To summarize, there is a natural identification between the tangent space and the cotangent space given by:
\begin{equation}
    \fun{T_\Sigma\SPD(n)}{T^*_\Sigma\SPD(n)}{X}{(Y\in T_\Sigma\SPD(n)\lmto\tr(d_\Sigma\id(X)d_\Sigma\id(Y)))}
\end{equation}
Hence, a cometric on SPD matrices can be seen as a metric.

Back to kernel metrics, it is interesting to note that this class is stable under taking the cometric and that the cometric has a simple expression.

\begin{proposition}[Cometric stability of kernel metrics]\label{prop:cometric_stability}
Let $g$ be a kernel metric with kernel function $\phi$. Then the cometric $g^*$ seen as a metric through the identification explained above is a kernel metric with kernel function $\phi^*=1/\phi$.
\end{proposition}

This elementary fact is interesting from a numerical point of view. Indeed, to compute numerically the geodesics, one can either integrate the geodesic equation involving the Christoffel symbols (which is of second order) or integrate its Hamiltonian version involving the cometric (which is of first order). Hence, the fact that the cometric of a kernel metric is available is a quite important result that appeared to be previously unnoticed. More precisely, the geodesic equation writes $\ddot{x}^k+\Gamma_{ij}^k\dot{x}^i\dot{x}^j=0$ where $x(t)$ is a curve on the manifold $\M$ and $\Gamma_{ij}^k$ are the Christoffel symbols related to the metric by $\Gamma_{ij}^k=\frac{1}{2}g^{kl}(\partial_ig_{jl}+\partial_jg_{il}-\partial_lg_{ij})$. By considering a curve $p(t)$ on the cotangent bundle $T^*\M$ instead, and $x(t)$ the curve on the manifold $\M$ such that $p(t)\in T^*_{x(t)}\M$, the geodesic equation admits the following Hamiltonian formulation:
\begin{equation}
    \twoeq{\dot{x}^k=g^{kl}p_l}{\dot{p}_l=-\frac{1}{2}\frac{\partial g^{ij}}{\partial x^l}p_ip_j}.
\end{equation}
The Hamiltonian equation is often preferred to compute the geodesics numerically since the integration is simpler and more stable. It only involves the cometric, which is very easy to compute for a kernel metric.

\subsubsection{Canonical Frobenius-like expression of a kernel metric}\label{subsubsec:expression_kernel_metric}

An expression of kernel metrics was given in \cite{Hiai09} by means of the operators $\mathbb{L}_\Sigma:X\lmto\Sigma X$, $\mathbb{R}_\Sigma:X\lmto X\Sigma$ and $\phi(\mathbb{L}_\Sigma,\mathbb{R}_\Sigma):\Sym(n)\lto\Sym(n)$ defined for $\Sigma=PDP^\top\in\SPD(n)$ by $[\phi(\mathbb{L}_\Sigma,\mathbb{R}_\Sigma)X]_{ij}=\phi(d_i,d_j)[P^\top XP]_{ij}$. This expression is $g^\phi_\Sigma(X,X)=\tr(X\phi(\mathbb{L}_\Sigma,\mathbb{R}_\Sigma)^{-1}(X))$. Beware that ``$\phi(\mathbb{L}_\Sigma,\mathbb{R}_\Sigma)$" is just a notation, it is not a strict composition between $\phi$ and the operators $\mathbb{L}_\Sigma$ and $\mathbb{R}_\Sigma$. The existence of the map $\Phi:\SPD(n)\times\Sym(n)\lto\Sym(n)$ hidden in $\phi(\mathbb{L}_\Sigma,\mathbb{R}_\Sigma):=\Phi_\Sigma$ is ensured by Lemma \ref{lemma:spectral_dec} by extending the $\Orth(n)$-equivariant map $\Phi:\Diag^+(n)\times\Sym(n)\lto\Sym(n)$ defined by $[\Phi_D(X)]_{ij}=\phi(d_i,d_j)X_{ij}$. In this work, we even prefer to define the bivariate map $\psi=\phi^{-1/2}$ and define in a analogous way the map $\Psi:\SPD(n)\times\Sym(n)\lto\Sym(n)$ so that we can write the kernel metric with a suitable Frobenius-like expression:
\begin{equation}
    g^\phi_\Sigma(X,X)=\tr(\Psi_\Sigma(X)^2)
\end{equation}
We can give explicitly $\Psi$ in some particular cases:
\begin{enumerate}
    \itemsep0em
    \item Euclidean metric: $\Psi^\E_\Sigma(X)=X$;
    \item log-Euclidean metric: $\Psi^\LE_\Sigma(X)=d_\Sigma\log(X)$;
    \item affine-invariant metric: $\Psi^\A_\Sigma(X)=\Sigma^{-1/2}X\Sigma^{-1/2}$.
\end{enumerate}
This is an important step towards the trace term extension.

\subsubsection{Kernel metrics with a trace term}\label{subsubsec:kernel_trace_term}

The class of kernel metrics does not encompass the $\Orth(n)$-invariant Euclidean, $\Orth(n)$-invariant log-Euclidean and affine-invariant metrics with a trace factor $\beta\ne 0$. However, thanks to the previous canonical expression, we can define a natural extension of a kernel metric with a trace term.

\begin{definition}[Extended kernel metrics]\label{def:kernel_scaling_trace}
Let $g^\phi$ be a kernel metric associated to the kernel function $\phi:(0,\infty)^2\lto(0,\infty)$. We define the map $\psi=\phi^{-1/2}$ and the map $\Psi:\SPD(n)\times\Sym(n)\lto\Sym(n)$ as described above so that $g_\Sigma(X,X)=\tr(\Psi_\Sigma(X)^2)$. We define a two-parameter family which extends the kernel metric $g^\phi$ for all $\Sigma\in\SPD(n)$ and $X\in\Sym(n)$ by:
\begin{equation}
    g^{\phi,\alpha,\beta}_\Sigma(X,X)=\alpha\,\tr(\Psi_\Sigma(X)^2)+\beta\,\tr(\Psi_\Sigma(X))^2
\end{equation}
where $(\alpha,\beta)\in\mathbf{ST}$, i.e. $\alpha>0$ and $\alpha+n\beta>0$.
\end{definition}

It can be shown that for the Bures-Wasserstein and the BKM metrics, the trace term such defined would be $\beta\,\tr(\Sigma^{-1/2}X)^2$. Contrarily to the log-Euclidean and the affine-invariant cases, there is no isometry a priori between two metrics of the family. It is interesting to note that Propositions \ref{prop:main_properties_kernel_metrics} and \ref{prop:cometric_stability} are still valid for these extended kernel metrics. We omit the proofs since they are analogous to the ones given for kernel metrics in \cite{Hiai09}.

\begin{proposition}[Key results on extended kernel metrics]\label{prop:key_results_kernel_st}
\begin{enumerate}
    \itemsep0em
    \item[]
    \item (Generality) All the metrics in Section \ref{sec:inventory} are extended kernel metrics.
    \item (Stability) The class of extended kernel metrics is stable under univariate diffeomorphisms and the transformation is the same as in Proposition \ref{prop:main_properties_kernel_metrics}.
    \item (Completeness) An extended mean kernel metric with homogeneity power $\theta$ is geodesically complete if and only if $\theta=2$.
    \item (Cometric) The class of extended kernel metrics is cometric-stable and the corresponding transformation is $(\phi,\alpha,\beta)\lmto(\frac{1}{\phi},\frac{1}{\alpha},-\frac{\beta}{\alpha(\alpha+n\beta)})$.
\end{enumerate}
\end{proposition}

In this section, we recalled the definition of kernel metrics and three key properties. We added the property of stability under the cometric with an explicit expression and we argued that it is an interesting property from a numerical point of view to compute geodesics. We found a wider class of metrics which satisfies the same key properties and which encompasses all the $\Orth(n)$-invariant metrics defined in Section \ref{sec:inventory}. It is now tempting to look for wider classes of $\Orth(n)$-invariant metrics and to determine if these properties are still valid.

In the next section, we characterize $\Orth(n)$-invariant metrics by means of three multivariate functions satisfying conditions of compatibility, positivity and symmetry. This result allows to understand better the specificity of kernel metrics and extended kernel metrics within the whole class of $\Orth(n)$-invariant metrics. Then we give a counterpart of Proposition \ref{prop:key_results_kernel_st} and we propose a new intermediate class of $\Orth(n)$-invariant metrics which is cometric stable.

\section{Characterization of $\mathrm{O}(n)$-invariant metrics}\label{sec:characterization}

In this section, we give a characterization of $\Orth(n)$-invariant metrics on SPD matrices. We present it as an extension of Theorem \ref{thm:characterization_o(n)-invariant_inner_products} characterizing $\Orth(n)$-invariant inner products on symmetric matrices. Instead of two parameters $\alpha,\beta$ which satisfy a positivity condition, an $\Orth(n)$-invariant metric is characterized by three multivariate functions $\alpha,\beta,\gamma:(0,\infty)^n\lto\R$ which satisfy a positivity condition plus a compatibility condition and a symmetry condition. This is explained in Section \ref{subsec:theorem}. We also give two corollary results which characterize two subclasses of $\Orth(n)$-invariant metrics with additional invariances: scaling invariance and inverse-consistency. Section \ref{subsec:proof_theorem} is dedicated to the proof of the theorem. In Section \ref{subsec:reinterpretation}, we reinterpret kernel metrics in light of the theorem. In Section \ref{subsec:key_results_o(n)-invariant}, we give key results on $\Orth(n)$-invariant metrics and we compare them to those on kernel metrics given in Proposition \ref{prop:main_properties_kernel_metrics}. In particular, we state that the cometric can be difficult to compute. Hence in Section \ref{subsec:bivariate_separable}, we introduce the class of bivariate separable metrics which is an intermediate class between $\Orth(n)$-invariant and extended kernel metrics, which is cometric-stable and for which the cometric is known in closed-form.

\subsection{Theorem and corollaries}\label{subsec:theorem}

Let us rephrase the characterization of $\Orth(n)$-invariant inner products on $\Sym(n)$ (Theorem \ref{thm:characterization_o(n)-invariant_inner_products}). An inner product $\dotprod{\cdot}{\cdot}$ on $\Sym(n)$ is $\Orth(n)$-invariant if and only if there exist real numbers $\gamma,\alpha>0$ and $\beta\in\R$ such that:
\begin{equation}
    \dotprod{X}{X}=\gamma\sum_i{X_{ii}^2}+\alpha\sum_{i\ne j}{X_{ij}^2}+\beta\sum_{i\ne j}{X_{ii}X_{jj}},
\end{equation}
\begin{enumerate}
    \itemsep0em
    \item (Compatibility) $\gamma=\alpha+\beta$,
    \item (Positivity) the symmetric matrix $S$ defined by $S_{ii}=\gamma$ and $S_{ij}=\beta$ is positive definite.
\end{enumerate}
The characterization of $\Orth(n)$-invariant metrics on $\SPD(n)$ has an analogous form where real numbers are replaced by $n$-multivariate functions and where there is an additional property of symmetry of these functions. We introduce this notion of symmetry before stating the theorem. The proof is in Section \ref{subsec:proof_theorem}.

\begin{definition}[$(k,n-k)$-symmetric functions]
We say that a function $f:(0,\infty)^n\lto\R$ is $(k,n-k)$-symmetric if it is symmetric in its $k$ first variables and symmetric in its $n-k$ last variables. In other words, $f$ is invariant under permutations $\sigma=\sigma_1\sigma_2$ where $\sigma_1$ has support in $\intg{1}{k}$ and $\sigma_2$ has support in $\intg{k+1}{n}$. Hence, given a set $I\subseteq\intg{1}{n}$ of cardinal $k$ and $d\in(0,\infty)^n$, we denote $f(d_{i\in I},d_{i\notin I}):=f(\sigma\cdot d)$ where $\sigma(\intg{1}{k})=I$ and $(\sigma\cdot d)_i=d_{\sigma(i)}$.
\end{definition}

\begin{theorem}[Characterization of $\Orth(n)$-invariant metrics]\label{thm:characterization_o(n)-invariant_metrics}
Let $g$ be a Riemannian metric on $\SPD(n)$. If $g$ is $\Orth(n)$-invariant, then there exist three maps $\gamma,\alpha:(0,\infty)^n\lto(0,\infty)$ and $\beta:(0,\infty)^n\lto\R$ such that for all $\Sigma=PDP^\top\in\SPD(n)$ and $X=PX'P^\top\in T_\Sigma\SPD(n)$:
\begin{align}
    &g_\Sigma(X,X) = g_D(X',X') \label{align_theorem}\\
    &= \sum_i{\gamma(d_i,d_{k\ne i})X_{ii}'^2}+\sum_{i\ne j}{\alpha(d_i,d_j,d_{k\ne i,j})X_{ij}'^2}+\sum_{i\ne j}{\beta(d_i,d_j,d_{k\ne i,j})X_{ii}'X_{jj}'}, \nonumber
\end{align}
\begin{enumerate}
    \itemsep0em
    \item (Compatibility) $\gamma$ equals $\alpha+\beta$ on the set $\mc{D}=\{d\in(0,\infty)^n|d_1=d_2\}$,
    \item (Positivity) for all $d\in(0,\infty)^n$, the symmetric matrix $S(d)$ defined by $S_{ii}(d)=\gamma(d_i,d_{k\ne i})$ and $S_{ij}(d)=\beta(d_i,d_j,d_{k\ne i,j})$ is positive definite,
    \item (Symmetry) $\gamma$ is $(1,n-1)$-symmetric and $\alpha,\beta$ are $(2,n-2)$-symmetric.
\end{enumerate}
Conversely, if there exist such maps $\alpha,\beta,\gamma$, then Equation (\ref{align_theorem}) correctly defines an $\Orth(n)$-invariant Riemannian metric.\\
Moreover, $g$ is continuous if and only if $\alpha,\beta,\gamma$ are continuous.
\end{theorem}

Before giving the proof, we observe that this theorem allows to characterize subclasses of $\Orth(n)$-invariant metrics as well. Here we give the general form of $\Orth(n)$-invariant metrics that are invariant under scaling and under inversion respectively. We omit the proof.

\begin{proposition}[Characterizations of subclasses of $\Orth(n)$-invariant metrics]
Let $g$ be an $\Orth(n)$-invariant metric characterized by the maps $\alpha,\beta,\gamma$.
\begin{enumerate}
    \item $g$ is invariant under scaling if and only if $f(\lambda x)=\frac{1}{\lambda^2}f(x)$ for $f\in\{\alpha,\beta,\gamma\}$, for all $x\in(0,\infty)^n$ and for all $\lambda>0$.
    \item $g$ is invariant under inversion if and only if $\gamma(d_1^{-1},...,d_n^{-1})=d_1^4\,\gamma(d_1,...,d_n)$ and $f(d_1^{-1},...,d_n^{-1})=d_1^2d_2^2\,f(d_1,...,d_n)$ for $f\in\{\alpha,\beta\}$, for all $d\in(0,\infty)^n$.
\end{enumerate}
\end{proposition}

\subsection{Proof of the theorem}\label{subsec:proof_theorem}

\begin{proof}[Proof of Theorem \ref{thm:characterization_o(n)-invariant_metrics} (Characterization of $\Orth(n)$-invariant metrics)]
Let $g$ be an $\Orth(n)$-invariant metric on $\SPD(n)$. Since any diagonal matrix $D$ is invariant under the subgroup $\mc{D}^\pm(n)$, the inner product $g_D$ is $\mc{D}^\pm(n)$-invariant. Hence, Lemma \ref{lemma:characterization_inner_products} \ref{enum:characterization_d(n)-invariant} ensures that there are positive coefficients $\alpha_{ij}(D)=\alpha_{ji}(D)$ and a matrix $S(D)\in\SPD(n)$ s.t. $g_D(X,X)=\sum_{i\ne j}{\alpha_{ij}(D)X_{ij}^2}+\sum_{i,j}{S_{ij}(D)X_{ii}X_{jj}}$. Then, we define the three maps:
\begin{enumerate}[label=$\bullet$]
    \item $\alpha:d\in(0,\infty)^n\lmto\alpha_{12}(\Diag(d))>0$,
    \item $\beta:d\in(0,\infty)^n\lmto S_{12}(\Diag(d))$,
    \item $\gamma:d\in(0,\infty)^n\lmto S_{11}(\Diag(d))>0$.
\end{enumerate}
Following the same idea as in the proof of Lemma \ref{lemma:characterization_inner_products} \ref{enum:characterization_s*(n)-invariant}, we use the invariance under permutations since $\Diag^+(n)$ is stable under this action. Then, one easily checks that $\alpha,\beta$ are $(2,n-2)$-symmetric and $\gamma$ is $(1,n-1)$-symmetric and that we can express the other coefficients in function of $\alpha,\beta,\gamma$ by permuting the $d_i$'s. We get for $i\ne j$:
\begin{enumerate}[label=$\bullet$]
    \item $\alpha_{ij}(\Diag(d))=\alpha(d_i,d_j,d_{k\ne i,j})$
    \item $S_{ij}(\Diag(d))=\beta(d_i,d_j,d_{k\ne i,j})$
    \item $S_{ii}(\Diag(d))=\gamma(d_i,d_{k\ne i})$
\end{enumerate}
So we get the expression (\ref{align_theorem}), the symmetry and the positivity conditions. We only miss the compatibility condition so let $d=(d_1,...,d_n)\in(0,\infty)^n$ such that $d_1=d_2$. Since $D=\Diag(d)$ is stable under any block-diagonal orthogonal matrix $R=\Diag(R_\theta,I_{n-2})\in\Orth(n)$ with $R_\theta\in\Orth(2)$, with the same computations as in the proof of Theorem \ref{thm:characterization_o(n)-invariant_inner_products}, we get $\gamma(d)=\alpha(d)+\beta(d)$.\\

Conversely, if $\alpha,\beta,\gamma$ are three maps satisfying the conditions of compatibility, positivity and symmetry, then we define $g_D(X,X)=\sum_i{\gamma(d_i,d_{k\ne i})X_{ii}^2}+\sum_{i\ne j}{\alpha(d_i,d_j,d_{k\ne i,j})X_{ij}^2}+\sum_{i\ne j}{\beta(d_i,d_j,d_{k\ne i,j})X_{ii}X_{jj}}$. We have to show that defining $g_\Sigma(X,X)=g_D(P^\top XP,P^\top XP)$ does not depend on the chosen eigenvalue decomposition $\Sigma=PDP^\top\in\SPD(n)$. According to Lemma \ref{lemma:spectral_dec}, we have three cases to study. One can easily show that the only non-trivial case is the third one, involving a diagonal matrix $D=\Diag(\lambda_1I_{m_1},...,\lambda_pI_{m_p})$ with sorted diagonal values $\lambda_1>...>\lambda_p>0$ and a block-diagonal orthogonal matrix $R=\Diag(R_1,...,R_p)\in\Orth(n)$ with $R_k\in\Orth(m_k)$. So we have to show that $g_D(R^\top XR,R^\top XR)=g_D(X,X)$ for all matrix $X\in\Sym(n)$, since $R^\top DR=D$. We denote $\bar{X}^{kl}\in\Mat(m_k,m_l)$ the $(k,l)$ block matrix defined by $\bar{X}^{kl}_{ij}=X_{n_{k-1}+i,n_{l-1}+j}$ where $n_k=\sum_{j=1}^k{m_j}$. Note that $\bar{X}^{kk}\in\Sym(m_k)$ is the $k$-th diagonal block of $X$ and $\bar{X}^{lk}=(\bar{X}^{kl})^\top$. Therefore $\overline{R^\top XR}^{kl}=R_k^\top\bar{X}^{kl}R_l$. In the following, we split the sums between the blocks with multiplicity 1 and the blocks with higher multiplicity and we use the compatibility condition. The notation $\alpha(\lambda_k,\lambda_l,...)$ stands for $\alpha(d_i,d_j,d_{m\ne i,j})$ where $\lambda_k=d_i$ and $\lambda_l=d_j$, i.e. $n_{k-1}+1\ls i\ls n_k$ and $n_{l-1}+1\ls j\ls n_l$. We compute the difference:
\begin{align*}
    &g_D(R^\top XR,R^\top XR) - g_D(X,X)\\
    &\quad = \sum_{k|m_k=1}{\gamma(d_{n_k},d_{m\ne n_k})\underset{0}{\underbrace{((R^\top XR)_{n_kn_k}^2-X_{n_kn_k}^2)}}}\\
    &\quad\quad + \sum_{\substack{k\ne l \\ m_k=m_l=1}}{\alpha(\lambda_k,\lambda_l,...)\underset{0}{\underbrace{((R^\top XR)_{n_kn_l}^2-X_{n_kn_l}^2)}}}\\
    &\quad\quad + \sum_{\substack{k\ne l \\ m_k=m_l=1}}{\beta(\lambda_k,\lambda_l,...)\underset{0}{\underbrace{((R^\top XR)_{n_kn_k}(R^\top XR)_{n_ln_l}-X_{n_kn_k}X_{n_ln_l})}}}\\
    &\quad\quad + \sum_{k|m_k>1}{\underset{\alpha(\lambda_k,\lambda_k,...)+\beta(\lambda_k,\lambda_k,...)}{\underbrace{\gamma(\lambda_k,\lambda_k,...)}}\sum_{i=n_{k-1}+1}^{n_k}{((R^\top XR)_{ii}^2-X_{ii}^2)}}\\
    &\quad\quad + \sum_{\substack{k,l \\ m_k\,\mathrm{or}\,m_l>1}}{\alpha(\lambda_k,\lambda_l,...)\sum_{\substack{n_{k-1}+1\ls i\ls n_k \\ n_{l-1}+1\ls j\ls n_l \\ i\ne j}}{((R^\top XR)_{ij}^2-X_{ij}^2)}}\\
    &\quad\quad + \sum_{\substack{k,l \\ m_k\,\mathrm{or}\,m_l>1}}{\beta(\lambda_k,\lambda_l,...)\sum_{\substack{n_{k-1}+1\ls i\ls n_k \\ n_{l-1}+1\ls j\ls n_l \\ i\ne j}}{((R^\top XR)_{ii}(R^\top XR)_{jj}-X_{ii}X_{jj})}}.
\end{align*}

Hence the missing term $i=j$ in the two last sums is provided by the sum weighted by $\gamma$. After a change of indexes based on the equality $\overline{R^\top XR}^{kl}=R_k^\top\bar{X}^{kl}R_l$, we get:
\begin{align*}
    &g_D(R^\top XR,R^\top XR) - g_D(X,X)\\
    &\quad =\sum_{\substack{k,l \\ m_k\,\mathrm{or}\,m_l>1}}{\alpha(\lambda_k,\lambda_l,...)\underset{\tr(R_k^\top\bar{X}^{kl}R_l(R_k^\top\bar{X}^{kl}R_l)^\top)-\tr(\bar{X}^{kl}(\bar{X}^{kl})^\top)=0}{\underbrace{\sum_{i=1}^{m_k}\sum_{j=1}^{m_l}{((R_k^\top\bar{X}^{kl}R_l)_{ij}^2-(\bar{X}^{kl})_{ij}^2)}}}}\\
    &\quad\quad + \sum_{\substack{k,l \\ m_k\,\mathrm{or}\,m_l>1}}{\beta(\lambda_k,\lambda_l,...)\underset{\tr(R_k^\top\bar{X}^{kk}R_k)\tr(R_l^\top\bar{X}^{ll}R_l)-\tr(\bar{X}^{kk})\tr(\bar{X}^{ll})=0}{\underbrace{\sum_{i=1}^{m_k}\sum_{j=1}^{m_l}{((R_k^\top\bar{X}^{kk}R_k)_{ii}(R_l^\top\bar{X}^{ll}R_l)_{jj}-\bar{X}^{kk}_{ii}\bar{X}^{ll}_{jj})}}}}\\
    &\quad = 0.
\end{align*}

This proves that $g_\Sigma$ is well defined for all $\Sigma\in\SPD(n)$ and $\Orth(n)$-invariant by construction. The positivity condition ensures that $g$ is a metric.\\

Finally, it is clear that $\alpha,\beta,\gamma$ have at least the same regularity as the metric $g$ since they are coordinates of the map $D\in\Diag^+(n)\lmto g_D$. Let us prove that if $\alpha,\beta,\gamma$ are continuous, then $g$ is continuous. The main argument is in the following lemma (proved after the proof of the theorem).

\begin{lemma}[Continuity of eigenvalues and eigenvectors]\label{lemma:continuity_eig}
Let $\Sigma,\Lambda\in\Sym^+(n)$. Let $D,\Delta\in\Diag^+(n)$ be their matrices of ordered eigenvalues, i.e. $D=\Diag(d_1,...,d_n)$ and $\Delta=\Diag(\delta_1,...,\delta_n)$ with $d_1\ls...\ls d_n$ and $\delta_1\ls...\ls\delta_n$. Then, denoting $\|\cdot\|$ the Frobenius norm of matrices:
\begin{enumerate}
    \item $\|D-\Delta\|\ls\|\Sigma-\Lambda\|$,
    \item for all $Q\in\Orth(n)$ such that $\Lambda=Q\Delta Q^\top$, there exists $P\in\Orth(n)$ such that $\Sigma=PDP^\top$ and $\|P-Q\|^2\ls 4\sqrt{\frac{n}{m}}\|\Sigma-\Lambda\|$ where $m=\underset{\lambda\ne\mu\in\eig(\Sigma)}{\min}(\lambda-\mu)^2$.
\end{enumerate}
\end{lemma}
Let us prove that $g$ is continuous by showing that for all $\varepsilon$, for all $\Sigma,\Lambda\in\Sym^+(n)$, there exists $\eta$ such that if $\|\Sigma-\Lambda\|\ls\eta$, then for all $X\in\Sym^+(n)$, $|g_\Sigma(X,X)-g_\Lambda(X,X)|\ls\varepsilon\|X\|^2$. Let $\varepsilon>0$ and $\Sigma,\Lambda\in\Sym^+(n)$. Given Lemma \ref{lemma:continuity_eig}, let $D,\Delta\in\Diag^+(n)$ and $P,Q\in\Orth(n)$ such that $\Sigma=PDP^\top$, $\Lambda=Q\Delta Q^\top$, $\|D-\Delta\|\ls\|\Sigma-\Lambda\|$ and $\|P-Q\|^2\ls 4\sqrt{\frac{n}{m}}\|\Sigma-\Lambda\|$. For all $X\in\Sym(n)$:
\begin{align*}
    &|g_\Sigma(X,X)-g_\Lambda(X,X)|\ls\sum_i|\gamma(d_i,d_{k\ne i})[P^\top XP]_{ii}^2-\gamma(\delta_i,\delta_{k\ne i})[Q^\top XQ]_{ii}^2|\\
    &\quad +\sum_{i\ne j}|\alpha(d_i,d_j,d_{k\ne i,j})[P^\top XP]_{ij}^2-\alpha(\delta_i,\delta_j,\delta_{k\ne i,j})[Q^\top XQ]_{ij}^2|\\
    &\quad +\sum_{i\ne j}|\beta(d_i,d_j,d_{k\ne i,j})[P^\top XP]_{ii}[P^\top XP]_{jj}-\beta(\delta_i,\delta_j,\delta_{k\ne i,j})[Q^\top XQ]_{ii}[Q^\top XQ]_{jj}|
\end{align*}

To use Lemma \ref{lemma:continuity_eig}, we separate the eigenvalues and eigenvectors by introducing $0=-\gamma(d_i,d_{k\ne i})[Q^\top XQ]_{ii}^2+\gamma(d_i,d_{k\ne i})[Q^\top XQ]_{ii}^2$ in the absolute value on the first line, and analogous terms for $\alpha,\beta$. We get:
\begin{align*}
    |g_\Sigma(X,X)-g_\Lambda(X,X)|&\ls 3C\sum_{i,j,k,l}|[P^\top XP]_{ij}[P^\top XP]_{kl}-[Q^\top XQ]_{ij}[Q^\top XQ]_{kl}|\\
    &\quad +3\|X\|^2\max_{f\in\{\alpha,\beta,\gamma\}}\sum_{\sigma\in\mf{S}(n)}|f\circ\sigma(D)-f\circ\sigma(\Delta)|,
\end{align*}
where $C=\max_{f\in\{\alpha,\beta,\gamma\},\sigma\in\mf{S}(n)}|f\circ\sigma(D)|$.

Since $\alpha,\beta,\gamma$ and permutations are continuous, the term $\max_f\sum_\sigma|f\circ\sigma(D)-f\circ\sigma(\Delta)|$ can be made inferior than $\frac{\varepsilon}{6}$ for $\Delta$ sufficiently close to $D$, let's say $\|D-\Delta\|\ls\eta_1$ for a given $\eta_1>0$. On the other hand, for all $i,j,k,l\in\{1,...,n\}$:
\begin{align*}
    &\sum_{i,j,k,l}|[P^\top XP]_{ij}[P^\top XP]_{kl}-[Q^\top XQ]_{ij}[Q^\top XQ]_{kl}|\\
    &\quad\ls
    \sum_{i,j,k,l}|[P^\top XP]_{ij}|(|[P^\top XP]_{kl}-[P^\top XQ]_{kl}|+|[P^\top XQ]_{kl}-[Q^\top XQ]_{kl}|)\\
    &\quad\quad+(|[P^\top XP]_{ij}-[P^\top XQ]_{ij}|
    +|([P^\top XQ]_{ij}-[Q^\top XQ]_{ij})|)|[Q^\top XQ]_{kl}|\\
    &\quad\ls (\|P^\top XP\|_1+\|Q^\top XQ\|_1)(\|(P^\top X(P-Q)\|_1+\|(P-Q)^\top XQ\|_1)\\
    &\quad\ls 4n^2\|P-Q\|\|X\|^2.
\end{align*}
So for $\|P-Q\|\ls\frac{\varepsilon}{24n^2C}$ and $\|D-\Delta\|\ls\eta_1$, we have $|g_\Sigma(X,X)-g_\Lambda(X,X)|\ls\varepsilon\|X\|^2$. Thus if we choose $\eta:=\min(\eta_1,\frac{\sqrt{m}}{4\sqrt{n}}(\frac{\varepsilon}{24n^2C})^2)$, then if $\|\Sigma-\Lambda\|\ls\eta$, we have $|g_\Sigma(X,X)-g_\Lambda(X,X)|\ls\varepsilon\|X\|^2$, which proves the continuity.
\end{proof}

\begin{proof}[Proof of Lemma \ref{lemma:continuity_eig} (Continuity of eigenvalues and eigenvectors)]
Let $\Sigma=PDP^\top$, $\Lambda=Q\Delta Q^\top\in\Sym^+(n)$ with $D,\Delta$ sorted by increasing order.
\begin{enumerate}
    \item By squaring the inequality and developing the trace, we get that $\|D-\Delta\|\ls\|\Sigma-\Lambda\|$ if and only if $\tr(D\Delta)\gs\tr(DU\Delta U^\top)$ where $U=P^\top Q$. After noticing that $\tr(DU\Delta U^\top)=\sum_{i,j}[DS\Delta]_{ij}$ where $S$ is a bistochastic matrix defined by $S_{ij}=U_{ij}^2$, it suffices to prove that the maximum of the following function on bistochastic matrices, $F:S\lmto\sum_{i,j}[DS\Delta]_{ij}$, is $F(I_n)=\tr(D\Delta)$. Since the set of bistochastic matrices is the convex hull of permutation matrices and $F$ is linear, it suffices to show this on permutation matrices. Indeed, if $S=\sum{t_kP_{\sigma_k}}$ with $t_k\gs0$ and $\sum{t_k}=1$, then $F(S)=\sum{t_kF(P_{\sigma_k})}\ls\sum{t_k}\tr(D\Delta)=\tr(D\Delta)$.
    Let $\sigma\in\mf{S}(n)\backslash\{\id\}$. Then there exist $i<j$ such that $\sigma(i)>\sigma(j)$. Hence:
    \begin{align*}
        F(P_{\sigma\circ(i,j)})-F(P_\sigma)&=d_{\sigma(j)}\delta_i+d_{\sigma(i)}\delta_j-d_{\sigma(i)}\delta_i-d_{\sigma(j)}\delta_j\\
        &=(d_{\sigma(j)}-d_{\sigma(i)})(\delta_i-\delta_j)\gs 0
    \end{align*}
    Since we can decompose any permutation $\sigma$ into a product of transpositions, we can show by recurrence on the number of factors that $F(I_n)-F(P_\sigma)\gs 0$ for all permutations $\sigma\in\mf{S}(n)$.
    
    \item We denote $R=\Diag(R_1,...,R_p)\in\Orth(n)$ a block-diagonal orthogonal matrix with $R_j\in\Orth(m_j)$, where $m_1,...,m_p\in\N$ are the multiplicities of the eigenvalues of $D=\Diag(\lambda_1I_{m_1},...,\lambda_pI_{m_p})$. We are looking for $R$ such that $\|PR-Q\|^2\ls 4\sqrt{\frac{n}{m}}\|\Sigma-\Lambda\|$ with $m=\min_{i\ne j}(\lambda_i-\lambda_j)^2$. We denote $U=P^\top Q$ and $W=\Diag(W_1,...,W_p)$ the block-diagonal submatrix of $U$ where $W_j\in\Mat(m_j)$. Then we have:
    \begin{align}
        \|PR-Q\|^2&=2\tr(I_n-R^\top U)=2\tr(I_n-R^\top W)\ls 2\sqrt{n}\|I_n-R^\top W\|,\nonumber\\
        \|PR-Q\|^4&\ls 4n\|I_n-R^\top W\|^2=4n\tr(I_n+WW^\top-2WR^\top). \label{eq:majoration_eigenvectors}
    \end{align}
    We choose $R$ as the orthogonal factor in a polar decomposition of $W=SR$ where $S=\sqrt{WW^\top}$ is a symmetric positive semi-definite matrix. Since for all $j\in\{1,...,p\}$, $W_jW_j^\top\ls I_{m_k}$ for the Lowner order (because $W_j$ is a principal block of the orthogonal matrix $U$), we have $WW^\top\ls I_n$. Thus $S=\sqrt{WW^\top}\gs WW^\top$ since $\sqrt{x}\gs x$ for all $x\in[0,1]$. So $\tr(WR^\top)=\tr(S)\gs\tr(WW^\top)$. Back to Equation (\ref{eq:majoration_eigenvectors}):
    \begin{align*}
        \|PR-Q\|^4&\ls 4n\,\tr(I_n-WW^\top)=4n\sum_{d_i\ne d_j}U_{ij}^2\\
        &\ls\frac{4n}{m}\sum_{d_i\ne d_j}(d_i-d_j)^2U_{ij}^2=\frac{4n}{m}\|DU-UD\|^2=\frac{4n}{m}\|\Sigma-QDQ^\top\|^2,\\
        \|PR-Q\|^2&\ls2\sqrt{\frac{n}{m}}(\|\Sigma-\Lambda\|+\|Q(\Delta-D)Q^\top\|)\ls4\sqrt{\frac{n}{m}}\|\Sigma-\Lambda\|,
    \end{align*}
    which proves the result.
\end{enumerate}
\end{proof}

The smoothness seems to be more complicated to study. We suspect additional conditions of compatibility on the derivatives of the smooth maps $\alpha,\beta,\gamma$ at the singular set of SPD matrices with repeated eigenvalues in order to make the metric $g$ is smooth.

\subsection{Reinterpretation of kernel metrics}\label{subsec:reinterpretation}

This theorem allows to reinterpret kernel metrics. The curiosity of this theorem is the function $\gamma$  because we have no information on it as soon as the $d_i$'s are distinct. If $\alpha,\beta,\gamma$ do not depend on their $n-2$ last arguments, i.e. if they are \textit{bivariate}, then $\gamma$ does not depend on its second argument either and $\gamma(d_1)$ must be equal to $\alpha(d_1,d_1)+\beta(d_1,d_1)$. Hence $g_\Sigma(X,X)=\sum_{i,j}{\alpha(d_i,d_j)X_{ij}'^2}+\sum_{i,j}{\beta(d_i,d_j)X_{ii}'X_{jj}'}$ with $\alpha>0$ and $\alpha+n\beta>0$, which is much more tractable. Moreover, if $\beta=0$, then the quadratic form has a \textit{diagonal} expression (sum of squares $X_{ij}'^2$, no mixed terms $X_{ii}'X_{jj}'$) in the basis of matrices induced by the \textit{orthogonal} matrix $P\in\Orth(n)$ in the eigenvalue decomposition of $\Sigma$. We say that the metric is \textit{ortho-diagonal}.

To sum up, the subclass of kernel metrics has two fundamental properties: it is bivariate ($\alpha=\gamma-\beta=1/\phi$) and ortho-diagonal ($\beta=0$). This is the reason why we propose to designate kernel (resp. mean kernel) metrics as Bivariate Ortho-Diagonal or BOD metrics (resp. Mean Ortho-Diagonal or MOD metrics), as summarized in Table \ref{tab:names_kernel_metrics}. The natural extension of Definition \ref{def:kernel_scaling_trace} with the Scaling and Trace factors can be called BOST (and MOST) metrics.

\begin{table}[htbp]
    \centering
    \begin{tabular}{|c|c|}
    \hline
        Previous description  & New designation \\
        \hline
        Kernel metric & BOD metric\\
        Mean kernel metric & MOD metric\\
        Extended kernel metric & BOST metric\\
        Extended mean kernel metric & MOST metric\\
        \hline
    \end{tabular}
    \caption{Name correspondences for kernel metrics and sub/super-classes}
    \label{tab:names_kernel_metrics}
\end{table}

\subsection{Key results on $\Orth(n)$-invariant metrics}\label{subsec:key_results_o(n)-invariant}

In Section \ref{sec:kernel}, we gave four key results on BOD/MOD metrics in Propositions \ref{prop:main_properties_kernel_metrics} and \ref{prop:cometric_stability}, and four key results on BOST/MOST metrics in Proposition \ref{prop:key_results_kernel_st}. Here we give the counterpart of these propositions for $\Orth(n)$-invariant metrics.

\begin{proposition}[Key results on $\Orth(n)$-invariant metrics]
\begin{enumerate}
    \itemsep0em
    \item[]
    \item (Generality) The class of $\Orth(n)$-invariant metrics obviously contains the classes of BOD, MOD, BOST, MOST metrics, hence it contains all the metrics in Section \ref{sec:inventory}.
    \item (Stability) The class of $\Orth(n)$-invariant metrics is obviously stable by $\Orth(n)$-equivariant diffeomorphisms of $\SPD(n)$. Hence it is stable by univariate diffeomorphisms $f:\SPD(n)\lto\SPD(n)$ and in this case, the pullback metric $f^*g^{\alpha,\beta,\gamma}$ is characterized by the three maps:
    \begin{enumerate}
        \itemsep0em
        \item $\alpha_f:d\in(0,\infty)^n\lmto\frac{\alpha(f(d))}{f^{[1]}(d_1,d_2)^2}$,
        \item $\beta_f:d\in(0,\infty)^n\lmto\frac{\beta(f(d))}{f^{[1]}(d_1,d_2)^2}$,
        \item $\gamma_f:d\in(0,\infty)^n\lmto\frac{\gamma(f(d))}{f'(d_1)^2}$.
    \end{enumerate}
    \item (Completeness) Let $g=g^{\alpha,\beta,\gamma}$ be an $\Orth(n)$-invariant metric. We assume that $\alpha,\beta,\gamma$ satisfy a homogeneity property which is similar to the one assumed for mean kernel metrics: there exists $\theta\in\R$ such that for $f\in\{\alpha,\beta,\gamma\}$, $x\in(0,\infty)^n$ and $\lambda>0$, we have $f(\lambda x)=\lambda^{-\theta}f(x)$. If the metric $g$ is geodesically complete, then $\theta=2$.
    \item (Cometric) The class of $\Orth(n)$-invariant metrics is obviously cometric-stable. The cometric is characterized by $\alpha^*=1/\alpha$ and $S^*=S^{-1}$ where $S(d)\in\SPD(n)$ is defined by $S_{ij}(d)=\beta(d_i,d_j,d_{k\ne i,j})$ and $S_{ii}(d)=\gamma(d_i,d_{k\ne i,j})$ for all $d\in(0,\infty)$ and $i\ne j$.
\end{enumerate}
\end{proposition}

We omit the proof since it consists in elementary verifications for all but the third statement, whose proof is analogous to the one given in \cite{Hiai09}.

About completeness, the result is much weaker for general $\Orth(n)$-invariant metrics. Indeed, we lost the converse sense: ``if $\theta=2$, then the metric is geodesically complete". According to the proof of \cite{Hiai09}, the key element to prove this converse sense is exactly the bivariance, plus the fact that a symmetric homogeneous mean satisfies $m(x,x)=x$. It is worth noticing that the direct sense is still true though.

About the cometric, we lost the closed-form expression we had for BOD and BOST metrics. Computing the cometric is numerically quite heavy in general because it is equivalent to invert the matrix $S(d)$ for all $d\in(0,\infty)^n$. However, note that when $\beta=0$, the cometric is obviously given by the triple $(1/\alpha,0,1/\gamma)$. These ortho-diagonal metrics can be seen as the multivariate generalization of BOD metrics. In the next section, we give a cometric-stable extension of the class of BOST metrics for which the cometric can be computed in closed form: the class of bivariate separable metrics.

\subsection{Bivariate separable metrics}\label{subsec:bivariate_separable}

We argued in Section \ref{subsec:reinterpretation} that \textit{bivariate} metrics are of the form $g_\Sigma(X,X)=\sum_{i,j}{\alpha(d_i,d_j)X_{ij}'^2}+\sum_{i,j}{\beta(d_i,d_j)X_{ii}'X_{jj}'}$ with $\alpha>0$ and $\alpha+n\beta>0$. Then, the first term corresponds to a BOD metric and it can be rewritten $\tr(\Psi_\Sigma(X)^2)$, but it is still difficult to write the second term in a more compact way. If the function $\beta$ is {\em separable}, i.e. if $\beta$ can be written $\beta(x,y)=\psi^{(1)}(x)\psi^{(2)}(y)$, then the second term is simply $\tr(\Psi^{(1)}_\Sigma(X))\tr(\Psi^{(2)}_\Sigma(X))$. Indeed, we can define $\Psi^{(k)}_D(X)=\Diag(\psi^{(k)}(d_i)X_{ii})$ and extend it into $\Psi^{(k)}_\Sigma$ as explained in Section \ref{subsubsec:expression_kernel_metric}. In particular, BOST metrics correspond to the case when $\beta(x,y)=\lambda\sqrt{\alpha(x,x)\alpha(y,y)}$ with $1+n\lambda>0$. The wider class of \textit{bivariate separable} metrics is actually cometric-stable and the cometric can be computed quite easily. This is stated in Proposition \ref{prop:bivariate_separable}.

\begin{proposition}[Cometric of bivariate separable metrics]\label{prop:bivariate_separable}
Let $\psi:(0,\infty)^2\lto(0,\infty)$ be a symmetric map and let $\psi^{(1)},\psi^{(2)}:(0,\infty)\lto(0,\infty)$ be two maps on positive real numbers. As explained above, we define their extensions $\Psi,\Psi^{(1)},\Psi^{(2)}:\SPD(n)\times\Sym(n)\lto\Sym(n)$. The quadratic form defined by $g_\Sigma(X,X)=\tr(\Psi_\Sigma(X)^2)+\tr(\Psi^{(1)}_\Sigma(X))\tr(\Psi^{(2)}_\Sigma(X))$ automatically satisfies the compatibility and symmetry conditions. Then $g$ is positive definite if and only if the vectors $x=x(d)=\left(\frac{\psi^{(1)}(d_i)}{\psi(d_i,d_i)}\right)_{1\ls i\ls n}$ and $y=y(d)=\left(\frac{\psi^{(2)}(d_i)}{\psi(d_i,d_i)}\right)_{1\ls i\ls n}$ satisfy the inequality $\|x\|\|y\|-\dotprod{x}{y}<2$ for all $d\in(0,\infty)^n$.

In this case, we say that $g$ is a Bivariate Separable metric. It is characterized by the matrix $S=\Delta(I_n+\frac{1}{2}(xy^\top+yx^\top))\Delta$ with $\Delta=\Diag(\psi(d_i,d_i))$. This class of metrics is cometric-stable and the cometric is given by:
\begin{equation}\label{eq:cometric}
    S^{-1}=\Delta^{-1}\left[I_n-\frac{1}{4c}(2+\dotprod{x}{y})(xy^\top+yx^\top)+\frac{1}{4c}(\|y\|^2xx^\top+\|x\|^2yy^\top)\right]\Delta^{-1}
\end{equation}
with $c=1+\dotprod{x}{y}-\frac{1}{4}(\|x\|^2\|y\|^2-\dotprod{x}{y}^2)$.
\end{proposition}

\begin{proof}[Proof of Proposition \ref{prop:bivariate_separable}]
To determine when $g$ is a metric, we express $\alpha,\beta,\gamma,S$ in function of $\psi,\psi^{(1)},\psi^{(2)}$:
\begin{enumerate}
    \item $\alpha(d_1,...,d_n)=\psi(d_1,d_2)^2>0$,
    \item $\beta(d_1,...,d_n)=\frac{1}{2}(\psi^{(1)}(d_1)\psi^{(2)}(d_2)+\psi^{(1)}(d_2)\psi^{(2)}(d_1))$,
    \item $\gamma(d_1,...,d_n)=\psi(d_1,d_1)^2+\psi^{(1)}(d_1)\psi^2(d_1)$,
    \item hence $S_{ij}(d)=\Delta^2_{ij}+\frac{1}{2}(\psi^{(1)}(d_i)\psi^{(2)}(d_j)+\psi^{(2)}(d_i)\psi^{(1)}(d_j))$, so we have $S=\Delta(I_n+\frac{1}{2}(xy^\top+yx^\top))\Delta$ with the notations of the proposition.
\end{enumerate}
The compatibility and symmetry conditions are trivially satisfied. The positivity condition reduces to $S\in\SPD(n)$, i.e. $I_n+\frac{1}{2}(xy^\top+yx^\top)\in\SPD(n)$. As the eigenvalues of $M=xy^\top+yx^\top$ are $0$ (with multiplicity $n-2$) and $\dotprod{x}{y}\pm\|x\|\|y\|$, $S$ is positive definite if and only if $2+\dotprod{x}{y}\pm\|x\|\|y\|>0$. But $\dotprod{x}{y}+\|x\|\|y\|\gs 0$ so there is only one condition: $2>\|x\|\|y\|-\dotprod{x}{y}(>0)$, as announced.

Now, we want to compute $S^{-1}$. As $M$ is of rank 2 at most, there exists a polynomial $P$ of degree 3 at most such that $P(I_n+\frac{1}{2}M)=0$. Let us find such a polynomial to compute $S^{-1}$. Since $M^2=\dotprod{x}{y}M+N$ with $N=\|y\|^2xx^\top+\|x\|^2yy^\top$ and $NM=\|x\|^2\|y\|^2M+\dotprod{x}{y}N$, we have:
\begin{align*}
    &\left(I_n+\frac{1}{2}M\right)^2=I_n+\left(1+\frac{\dotprod{x}{y}}{4}\right)M+\frac{1}{4}N,\\
    &\left(I_n+\frac{1}{2}M\right)^3=\left(I_n+\frac{1}{2}M\right)^2+\frac{1}{2}\left(I_n+\frac{1}{2}M\right)^2M\\
    &\quad\quad=\left(I_n+\frac{1}{2}M\right)^2+\frac{1}{2}M+\frac{1}{2}\left(1+\frac{\dotprod{x}{y}}{4}\right)M^2+\frac{1}{8}NM\\
    &\quad\quad=\left(I_n+\frac{1}{2}M\right)^2+\frac{4+4\dotprod{x}{y}+\dotprod{x}{y}^2+\|x\|^2\|y\|^2}{8}M+\frac{1}{4}(2+\dotprod{x}{y})N\\
    &\quad\quad=a\left(I_n+\frac{1}{2}M\right)^2+\frac{b}{2}M-(2+\dotprod{x}{y})I_n\\
    &\quad\quad=a\left(I_n+\frac{1}{2}M\right)^2+b\left(I_n+\frac{1}{2}M\right)+c\,I_n.
\end{align*}
with: $\accol{a=3+\dotprod{x}{y}}{b=\frac{-12-8\dotprod{x}{y}-\dotprod{x}{y}^2+\|x\|^2\|y\|^2}{4}}{c=1+\dotprod{x}{y}+\frac{\dotprod{x}{y}^2-x\|^2\|y\|^2}{4}=1-a-b}$.\\
Hence, denoting $S_0:=I_n+\frac{1}{2}M$, we have $S_0^{-1}=\frac{1}{c}\left(S_0^2-a\,S_0-bI_n\right)=I_n+\frac{1}{4c}(N-(2+\dotprod{x}{y})M)$ and $S^{-1}=\Delta^{-1}\left(I_n+\frac{1}{4c}(N-(2+\dotprod{x}{y})M)\right)\Delta^{-1}$ which is exactly Equation (\ref{eq:cometric}).

Finally, we want to prove that the cometric is bivariate separable. The case $y=0$ corresponds to a BOD metric so we can assume $y\ne 0$. Regarding Equation (\ref{eq:cometric}), we look for $x'=\frac{Ax+By}{4c}$ and $y'=Cx+Dy$ for $A,B,C,D\in\R$ such that:
\begin{equation}
    x'y'^\top+y'x'^\top=-\frac{1}{2c}(2+\dotprod{x}{y})(xy^\top+yx^\top)+\frac{1}{2c}(\|y\|^2xx^\top+\|x\|^2yy^\top)
\end{equation}
It is satisfied if $AC=\|y\|^2$, $BD=\|x\|^2$ and $AD+BC=-2(2+\dotprod{x}{y})$, or equivalently $(AX+B)(CX+D)=\|y\|^2X^2-2(2+\dotprod{x}{y})X+\|x\|^2$. This is a second-order polynomial with roots $\lambda=\frac{2+\dotprod{x}{y}+\sqrt\delta}{\|y\|^2}$ and $\mu=\frac{2+\dotprod{x}{y}-\sqrt\delta}{\|y\|^2}$ where $\delta=(2+\dotprod{x}{y}+\|x\|\|y\|)(2+\dotprod{x}{y}-\|x\|\|y\|)>0$ is the discriminant. Hence, it suffices to define $A=\|y\|$, $B=-\lambda\|y\|$, $C=\|y\|$ and $D=-\mu\|y\|$, so that $S^{-1}=\Delta^{-1}\left(I_n+\frac{1}{2}(x'y'^\top+y'x'^\top)\right)\Delta^{-1}$. Hence, the cometric is bivariate separable and this class of metrics is cometric-stable.
\end{proof}

\section{Conclusion}\label{sec:conclusion}

To encompass all the $\Orth(n)$-invariant metrics summarized in Section \ref{sec:inventory}, including the ones with a trace term ($\beta\ne 0$), we defined the class of extended kernel metrics. This class satisfies the key results of stability and completeness we selected from \cite{Hiai09} plus the cometric-stability with cometric in closed form, which is important to compute geodesics numerically via the Hamiltonian formulation. Then, from the characterization of $\Orth(n)$-invariant metrics in terms of three continuous maps $\alpha,\beta,\gamma:(0,\infty)^n\lto(0,\infty)$ satisfying properties of compatibility, positivity and symmetry, we were able to characterize kernel metrics as Bivariate Ortho-Diagonal (BOD) metrics. Among the key results on mean kernel metrics, the sufficient condition of completeness and the closed-form expression of the cometric disappear for general $\Orth(n)$-invariant metrics. We finally defined the intermediate class of bivariate separable metrics which is cometric-stable and for which the cometric has a simple expression.

Since kernel metrics encompass very different metrics regarding curvature and completeness, it would be nice to introduce some more requirements on metrics to perform the opposite work of defining principled sub-classes of (mean) kernel metrics. There is actually a companion paper in preparation on principled subfamilies of $\Orth(n)$-invariant metrics on SPD matrices where we propose such a framework. It would also be interesting to rely on the cometric-stability of kernel metrics or super-classes to effectively compute the geodesics numerically and to investigate their properties regarding statistical analyses. 

Another interesting direction would be to consider other properties of kernel metrics that were described in the original paper, namely monotonicity and comparison properties. It would be challenging to understand how they could be generalized to BOST metrics or even to $\Orth(n)$-invariant metrics. Furthermore, to our knowledge there is no trace of families of non $\Orth(n)$-invariant metrics in the literature. However, there exist some situations where the $\Orth(n)$-invariance is not relevant, for example on correlation matrices because the space is not stable under this group action. This a promising perspective for future works.

\section*{Acknowledgements}

This project has received funding from the European Research Council (ERC) under the European Union’s Horizon 2020 research and innovation program (grant G-Statistics agreement No 786854). This work has been supported by the French government, through the UCAJEDI Investments in the Future project managed by the National Research Agency (ANR) with the reference number ANR-15-IDEX-01 and through the 3IA Côte d’Azur Investments in the Future project managed by the National Research Agency (ANR) with the reference number ANR-19-P3IA-0002. The authors warmly thank Nicolas Guigui and Dimbihery Rabenoro for insightful discussions and careful proofreading of this manuscript.

\newpage
\appendix

\section{Proofs on the Bures-Wasserstein metric}\label{sec:proofs_bures-wasserstein}

\begin{proof}[Proof of Levi-Civita connection in Table \ref{tab:riemannian_operations_bures-wasserstein}]
Let $X,Y$ be vector fields on $\SPD(n)$. The Levi-Civita connection is computed in \cite{Malago18}. With our notation $X^0=\mc{S}_\Sigma(X)$ defined by $X=\Sigma X^0+X^0\Sigma$, their result writes $\nabla_XY=\partial_XY-\{X^0Y+Y^0X\}_S+\{\Sigma X^0Y^0+\Sigma Y^0X^0\}_S$ where $\{A\}_S=\frac{1}{2}(A+A^\top)$ is the symmetric part of the matrix $A$. It is easy to see that it rewrites $\nabla_XY = \partial_XY - (X^0\Sigma Y^0+Y^0\Sigma X^0)$ which is a simpler expression.

We would like to give a different proof that relies on the geometry of the horizontal distribution. According to \cite{ONeill66}, Lemma 1, $d\pi(\nabla^G_{X^h}Y^h)=\nabla_XY$, where $\nabla^G=\partial$ is the Levi-Civita connection of the Frobenius metric $G$ on $\GL(n)$, i.e. the derivative of coordinates in the canonical basis of matrices. We differentiate the equality $X^h=(X^0\circ\pi)\times\Id_{\GL(n)}$ on $\GL(n)$:
\begin{align*}
    (\nabla^G_{X^h}{Y^h})_{|A} &=\partial_{X^h_A}(Y^0\circ\pi)A+Y^0_{\pi(A)}X^h_A\\
    &=(\partial_{X_{\pi(A)}}Y^0)A+Y^0_{\pi(A)}X^0_{\pi(A)}A,\\
    (\nabla_XY)_{|AA^\top} &= d_A\pi((\nabla^G_{X^h}{Y^h})_{|A})\\
    &= AA^\top(\partial_{X_{\pi(A)}}Y^0)+(\partial_{X_{\pi(A)}}Y^0)AA^\top\\
    &\quad + AA^\top X^0_{\pi(A)}Y^0_{\pi(A)}+Y^0_{\pi(A)}X^0_{\pi(A)}AA^\top,\\
    (\partial_XY)_{|\Sigma} &=\Sigma(\partial_{X_\Sigma}Y^0)+(\partial_{X_\Sigma}Y^0)\Sigma+X_\Sigma Y^0_\Sigma+Y^0_\Sigma X_\Sigma\\
    &= \Sigma(\partial_{X_\Sigma}Y^0)+(\partial_{X_\Sigma}Y^0)\Sigma+\Sigma X^0_\Sigma Y^0_\Sigma+Y^0_\Sigma X^0_\Sigma\Sigma\\
    &\quad + X^0_\Sigma\Sigma Y^0_\Sigma+Y^0_\Sigma\Sigma X^0_\Sigma\\
    &=(\nabla_XY)_{|\Sigma} + X^0_\Sigma\Sigma Y^0_\Sigma+Y^0_\Sigma\Sigma X^0_\Sigma.
\end{align*}
Finally, we find $\nabla_XY = \partial_XY - (X^0\Sigma Y^0+Y^0\Sigma X^0)$ as expected.
\end{proof}

\begin{proof}[Proof of curvature in Table \ref{tab:riemannian_operations_bures-wasserstein}]
Let $X,Y\in T_\Sigma\SPD(n)$ be tangent vectors at $\Sigma\in\SPD(n)$. We would like to compute the sectional curvature $\kappa(X,Y)=\frac{R(X,Y,X,Y)}{\|X\|^2\|Y\|^2-\dotprod{X}{Y}^2}$, i.e. $R(X,Y,X,Y)$. Let $X^h,Y^h\in\mc{H}_{\Sigma^{1/2}}$ be the horizontal lifts of $X,Y$ at $\Sigma^{1/2}$ and $X^0,Y^0\in\Sym(n)$ defined as explained above. We extend $X^h,Y^h$ into vector fields by $X^h_A:=X^0A$ and $Y^h_A:=Y^0A$. We do so because the formula we use to compute the curvature is based on a Lie bracket and can only be computed with fields. As the curvature is a tensor, it only depends on the values of $X$ and $Y$ at $\Sigma$ so the way we extend the fields does not influence the result (but it simplifies the computation).

A first strategy to compute the curvature is to use the Levi-Civita connection via the definition $R(X,Y)Z=\nabla_X\nabla_YZ-\nabla_Y\nabla_XZ-\nabla_{[X,Y]}Z$. It is tedious but doable. Another one consists in using the relation between the curvatures of the quotient metric (here, Bures-Wasserstein) and the original metric (here, Frobenius) found in \cite{ONeill66}, formula \{4\}. According to this formula, since the Euclidean metric is flat, the formula is $R_\Sigma(X,Y,X,Y)=\frac{3}{4}\|\ver\llbracket X^h,Y^h\rrbracket_{\Sigma^{1/2}}\|^2$ where $\ver:X^v+X^h\in T\GL(n)\lmto X^v\in\mc{V}$ is the vertical projection and $\llbracket\cdot,\cdot\rrbracket$ denotes the Lie bracket on vector fields of $\GL(n)$, which must be distinguished from the matrix Lie bracket $[V,W]=VW-WV$. Note that the right term only depends on $X^h_{\Sigma^{1/2}}$ and $Y^h_{\Sigma^{1/2}}$ because if $f:\GL^+(n)\lto\R$ is a map, then $\ver\llbracket fX^h,Y^h\rrbracket_{\Sigma^{1/2}}=f(\Sigma^{1/2})\ver\llbracket X^h,Y^h\rrbracket_{\Sigma^{1/2}}+d_{\Sigma^{1/2}}f(Y^h)\underset{0}{\underbrace{\ver(X^h)}}$.

The rest of the proof consists in computing $\ver\llbracket X^h,Y^h\rrbracket=\llbracket X^h,Y^h\rrbracket-\hor\llbracket X^h,Y^h\rrbracket$. On the one hand, $\llbracket X^h,Y^h\rrbracket_A=Y^0X^h_A-X^0Y^h_A=-[X^0,Y^0]A$. On the other hand, let $Z^h_A:=\hor\llbracket X^h,Y^h\rrbracket_A=:Z^0_{AA^\top}A\in\mc{H}_A$. Now, we can fix $\Sigma\in\SPD(n)$ and $A=\Sigma^{1/2}$. We take a spectral decomposition $\Sigma=PDP^\top$ and we denote with a prime all the previous matrices taken in the basis $P$ of eigenvectors of $\Sigma$, e.g. ${X^0}'=P^\top X^0P$. Then:
\begin{align*}
    Z_\Sigma :=&~ d_{\Sigma^{1/2}}\pi(\llbracket X^h,Y^h\rrbracket_{\Sigma^{1/2}})=\Sigma[X^0,Y^0]-[X^0,Y^0]\Sigma,\\
    Z^h_{\Sigma^{1/2}} =&~ (d_{\Sigma^{1/2}}\pi_{|\mc{H}_{\Sigma^{1/2}}})^{-1}(d_{\Sigma^{1/2}}\pi(\llbracket X^h,Y^h\rrbracket_{\Sigma^{1/2}}))\\
    =&~ (d_{\Sigma^{1/2}}\pi_{|\mc{H}_{\Sigma^{1/2}}})^{-1}(Z_\Sigma),\\
    [Z^{0'}_{\Sigma}]_{ij} =&~ \frac{1}{d_i+d_j}(D[{X^0}',{Y^0}']-[{X^0}',{Y^0}']D)_{ij}\\
    =&~ \frac{d_i-d_j}{d_i+d_j}[{X^0}',{Y^0}']_{ij},\\
    [Z^{h'}_{\Sigma^{1/2}}]_{ij} =&~\sqrt{d_j}[Z^{0'}_{\Sigma}]_{ij}=\sqrt{d_j}\frac{d_i-d_j}{d_i+d_j}[{X^0}',{Y^0}']_{ij},\\
    (\ver\llbracket X^h,Y^h\rrbracket_{\Sigma^{1/2}})'_{ij} =&~(\llbracket X^h,Y^h\rrbracket_{\Sigma^{1/2}})'_{ij}-[Z^{h'}_{\Sigma^{1/2}}]_{ij}\\
    =&~ -[{X^0}',{Y^0}']_{ij}\sqrt{d_j}-\sqrt{d_j}\frac{d_i-d_j}{d_i+d_j}[{X^0}',{Y^0}']_{ij}\\
    =&~ -\frac{2d_i\sqrt{d_j}}{d_i+d_j}[{X^0}',{Y^0}']_{ij},\\
    R_\Sigma(X,Y,X,Y) =&~ \frac{3}{4}\|(\ver\llbracket X^h,Y^h\rrbracket_{\Sigma^{1/2}})'\|^2\\
    =&~ 3\sum_{i,j}{\frac{d_i^2d_j}{(d_i+d_j)^2}\left[{X^0}',{Y^0}'\right]_{ij}^2}\\
    =&~ \frac{3}{2}\sum_{i,j}{\frac{d_id_j}{d_i+d_j}\left[{X^0}',{Y^0}'\right]_{ij}^2},
\end{align*}
where $P^\top XP=D{X^0}'+{X^0}'D$ and $P^\top YP=D{Y^0}'+{Y^0}'D$.
\end{proof}

\begin{proof}[Proof of geodesic parallel transport between commuting matrices in Table \ref{tab:riemannian_operations_bures-wasserstein}]
We want to prove that the geodesic parallel transport of the Bures-Wasserstein metric between two commuting matrices is $\Pi_{\Sigma\to\Lambda}X=P\left[\sqrt{\frac{\delta_i+\delta_j}{d_i+d_j}}[P^\top XP]_{ij}\right]_{i,j}P^\top$ where $\Sigma=PDP^\top$ and $\Lambda=P\Delta P^\top\in\SPD(n)$. The geodesic parallel transport is $\Orth(n)$-invariant so we only need to prove that $[\Pi_{D\to\Delta}X]_{ij}=\sqrt{\frac{\delta_i+\delta_j}{d_i+d_j}}X_{ij}$. The geodesic from $D$ to $\Delta$ is $\gamma(t)=((1-t)\sqrt{D}+t\sqrt{\Delta})^2$. Let us define $X(t)=\left[\sqrt{\frac{((1-t)d_i+t\delta_i)^2+((1-t)d_j+t\delta_j)^2}{d_i+d_j}}X_{ij}\right]_{i,j}$ and let us check that $\nabla_{\dot\gamma}X=0$. We compute:
\begin{align*}
    [X^0(t)]_{ij}&=\frac{[X(t)]_{ij}}{\gamma_i(t)+\gamma_j(t)}=\frac{1}{\sqrt{d_i+d_j}\sqrt{((1-t)d_i+t\delta_i)^2+((1-t)d_j+t\delta_j)^2}}X_{ij},\\
    \dot\gamma(t)&=2(\sqrt{\Delta}-\sqrt{D})((1-t)\sqrt{D}+t\sqrt{\Delta}),\\
    {\dot\gamma}^0(t)&=\frac{1}{2}\dot\gamma(t)\gamma^{-1}(t)=\frac{1}{2}\gamma^{-1}(t)\dot\gamma(t)=(\sqrt{\Delta}-\sqrt{D})((1-t)\sqrt{D}+t\sqrt{\Delta})^{-1},\\
    [\dot{X}(t)]_{ij}&=\frac{2(\sqrt{\delta_i}-\sqrt{d_i})((1-t)d_i+t\delta_i)+2(\sqrt{\delta_j}-\sqrt{d_j})((1-t)d_j+t\delta_j)}{2\sqrt{d_i+d_j}\sqrt{((1-t)d_i+t\delta_i)^2+((1-t)d_j+t\delta_j)^2}}X_{ij}\\
    &=(\sqrt{\delta_i}-\sqrt{d_i})((1-t)d_i+t\delta_i)[X^0(t)]_{ij}\\
    &\quad +[X^0(t)]_{ij}(\sqrt{\delta_j}-\sqrt{d_j})((1-t)d_j+t\delta_j)\\
    &=[{\dot\gamma}^0(t)\gamma(t)X^0(t)+X^0(t)\gamma(t){\dot\gamma}^0(t)]_{ij},\\
    \nabla_{\dot\gamma(t)}X&=\dot{X}(t)-({\dot\gamma}^0(t)\gamma(t)X^0(t)+X^0(t)\gamma(t){\dot\gamma}^0(t))=0.
\end{align*}
So the geodesic parallel transport from $\Sigma=PDP^\top$ to $\Lambda=P\Delta P^\top$ is $\Pi_{\Sigma\to\Lambda}X=P\left[\sqrt{\frac{\delta_i+\delta_j}{d_i+d_j}}[P^\top XP]_{ij}\right]_{i,j}P^\top$.
\end{proof}

\begin{proof}[Proof of equation of the geodesic parallel transport in Table \ref{tab:riemannian_operations_bures-wasserstein}]
The geodesic parallel transport equation is $\nabla_{\dot\gamma(t)}X=0$ along the geodesic $\gamma(t)=\gamma^h(t)\gamma^h(t)^\top$ between $\Sigma$ and $\Lambda\in\SPD(n)$, where $\gamma^h(t)=(1-t)\Sigma^{1/2}+t\Sigma^{-1/2}(\Sigma^{1/2}\Lambda\Sigma^{1/2})^{1/2}$. For a vector field $X(t)$ on $\SPD(n)$ defined along $\gamma(t)$, we can define the horizontal lift $X^h(t)=X^0(t)\gamma^h(t)\in\mc{H}_{\gamma^h(t)}$ where $X^0(t)$ is defined by $X(t)=\gamma(t)X^0(t)+X^0(t)\gamma(t)$. We are going to prove that $X(t)$ is the geodesic parallel transport of $X\in T_\Sigma\SPD(n)$ if and only if $X^0(t)$ satisfies the following ODE:
\begin{equation}
    \gamma(t)\dot{X}^0(t)+\dot{X}^0(t)\gamma(t)+\gamma^h(t)\dot\gamma^{h\top} X^0(t)+X^0(t)\dot\gamma^h\gamma^h(t)^\top=0.
\end{equation}
To rewrite the geodesic parallel transport equation $\nabla_{\dot\gamma(t)}X=0$, we need to compute the following derivatives:
\begin{align*}
    \dot{X}(t)&=\gamma(t)\dot{X}^0(t)+\dot{X}^0(t)\gamma(t)+\dot\gamma(t)X^0(t)+X^0(t)\dot\gamma(t),\\
    \dot\gamma(t)&=\dot\gamma^h\gamma^h(t)^\top+\gamma^h(t)\dot\gamma^{h\top}~\mathrm{where}~    \dot\gamma^h=\dot\gamma^0(t)\gamma^h(t).
\end{align*}
Now, we simply rewrite the equation:
\begin{align*}
    \nabla_{\dot\gamma(t)}X=0 &\cns \dot{X}(t)-(\dot\gamma^0(t)\gamma(t)X^0(t)+X^0(t)\gamma(t)\dot\gamma^0(t))=0\\
    &\cns \gamma(t)\dot{X}^0(t)+\dot{X}^0(t)\gamma(t)\\
    &\quad\quad\quad +(\dot\gamma(t)-\dot\gamma^h\gamma^h(t)^\top)X^0(t)+X^0(t)(\dot\gamma(t)-\gamma^h(t)\dot\gamma^{h\top})=0\\
    &\cns \gamma(t)\dot{X}^0(t)+\dot{X}^0(t)\gamma(t)+\gamma^h(t)\dot\gamma^{h\top} X^0(t)+X^0(t)\dot\gamma^h\gamma^h(t)^\top=0.
\end{align*}
\end{proof}

\bibliographystyle{elsarticle-num-names}
\bibliography{biblio}

\end{document}